\documentclass[12pt]{article}

\usepackage{mathrsfs,amsthm,graphicx,color,verbatim,bbm,amsmath,amsfonts,amssymb,newclude,nicefrac,amsfonts,graphicx,geometry,enumerate,hyperref} 
\theoremstyle{plain}
\newtheorem{theorem}{Theorem}[section]
\newtheorem{lemma}[theorem]{Lemma}
\newtheorem{corollary}[theorem]{Corollary}
\newtheorem{proposition}[theorem]{Proposition}

\theoremstyle{definition}
\newtheorem{definition} [theorem]{Definition}

\begin{document}

\newcommand{\Lip}{\operatorname{Lip}}
\newcommand{\B}{\mathcal{B}}
\renewcommand{\O}{\mathcal{O}}
\renewcommand{\L}{\mathcal{L}}
\newcommand{\U}{\mathcal{U}}
\newcommand{\V}{\mathcal{V}}
\newcommand{\E}{\mathbb{E}}
\newcommand{\ES}{\mathbb{E}}
\renewcommand{\P}{\mathbb{P}}
\newcommand{\R}{\mathbb{R}}
\newcommand{\N}{\mathbb{N}}
\newcommand{\smallsum}{\textstyle\sum}
\newcommand{\tr}{\operatorname{trace}}
\newcommand{\citationand}{\&}
\newcommand{\dt}[1][t]{\, \mathrm{d} #1}
\newcommand{\grid}{\,\mathcal{P}}
\newcommand{\euler}{Z^{\,N}}
\newcommand{\leuler}{\tilde{Z}^{\,N}}
\newcommand{\exteuler}{\bar{Z}^{\,N}}

\newcommand{\floor}[2]{\lfloor #1\rfloor_{#2}}
\newcommand{\ceil}[2]{\lceil #1\rceil_{#2}}
\newcommand{\mesh}[2]{\|#1\|^{#2}_{\grid_T}}
\newcommand{\bigbrack}[2]{\big( #1\big)^{#2}}
\newcommand{\bignorm}[3]{\bnl #1\bnr^{#2}_{#3}}
\newcommand{\bigsharp}[2]{\big[ #1\big]^{#2}}
\newcommand{\lpn}[3]{\mathcal{L}^{#1}(#2;#3)}
\newcommand{\lpnb}[3]{L^{#1}(#2;#3)}
\newcommand{\eulerm}[1]{Z^{\,\Theta_{#1}}}
\newcommand{\eulerpart}[1]{Z^{#1}}

\newcommand{\expeuler}[1]{Z^{\,\text{exp},#1}}
\newcommand{\impeuler}[1]{Z^{\,\text{imp},#1}}

\newcommand{\embed}[2]{\kappa^I_{#1,\,#2}}

\newcommand{\groupC}{ C }
\newcommand{\diagC}{ C }
\newcommand{\psiC}{ \eta }
\newcommand{\driftC}{{\bf y}}
\newcommand{\diffusionC}{{\bf z}}
\newcommand{\power}{ q }

\newcommand{\set}{{\mathbb{H}}}

\newcommand{\resolvent}[2]{R_{#2}( #1 )}

\title{On the mild It\^o formula in Banach spaces}

\author{Sonja Cox, Arnulf Jentzen, Ryan Kurniawan, and Primo\v{z} Pu\v{s}nik
	 \\ 
\emph{University of Amsterdam, the Netherlands
	and}
\\
\emph{ETH Z\"urich, Switzerland}
} 
 
\maketitle

\begin{abstract}
The mild It\^o formula proposed in
Theorem 1 in [Da Prato, G., Jentzen, A., \& R\"ockner, M., 
A mild It\^o formula for SPDEs,
arXiv:1009.3526 (2012), To appear in the Trans.\ Amer.\ Math.\ Soc.]
has turned out to be a useful instrument to study
solutions and numerical approximations of stochastic
partial differential equations (SPDEs) which are
formulated as stochastic evolution equations
(SEEs) on Hilbert spaces.
In this article we generalize this mild It\^o formula
so that
it is applicable to solutions and numerical approximations of
SPDEs which are formulated as SEEs on UMD (unconditional martingale differences) Banach spaces.
This generalization is especially
useful for proving essentially sharp
weak convergence rates
for numerical approximations of SPDEs.
\end{abstract}

\tableofcontents

\section{Introduction}
\label{sec:intro}

The standard It\^o formula for finite dimensional
It\^o processes has been generalized in the literature
to infinite dimensions
so that it is applicable to It\^o processes
with values in infinite dimensional Hilbert or Banach spaces;
see Theorem~2.4 in Brze\'zniak,
Van Neerven, Veraar \& Weiss~\cite{BrzezniakVanNeervenVeraarWeis2008}.
This infinite dimensional
generalization of the standard
It\^o formula is, however,
typically not applicable to a solution
(or a numerical approximation)
of a stochastic partial differential equation (SPDE)
as solutions of SPDEs are often only solutions
in the mild or weak sense,
which are not It\^o processes
on the considered state space of the SPDE.
To overcome this lack of regularity of solutions
of SPDEs,
Da Prato et al.\ proposed
in Theorem~1
in~\cite{DaPratoJentzenRoeckner2012}
(see also~\cite[Section~5]{JentzenKurniawan2015arXiv})
an alternative formula which 
Da Prato et al.\ refer to
as a mild It\^o formula.
The mild It\^o formula in Theorem~1 in~\cite{DaPratoJentzenRoeckner2012}
is
(even in finite dimensions)
different to the standard It\^o formula but it applies
to the class of Hilbert space valued mild It\^o processes
which is a rather general class
of Hilbert space valued stochastic processes that includes
standard It\^o processes as well as
mild solutions and numerical approximations
of semilinear SPDEs as special cases.
In this work we generalize
the mild It\^o formula
so that it is applicable to mild It\^o processes
which take values in UMD (unconditional martingale differences) 
Banach spaces
with type~2;
see Definition~\ref{def:mildIto}
in Subsection~\ref{sec:mildprocesses},
see Theorem~\ref{thm:ito}
in Subsection~\ref{sec:mildito},
and see Corollary~\ref{cor:itoauto2}
in Subsection~\ref{sec:mildito} below.
This generalization of the mild It\^o formula
is especially useful for proving essentially
sharp weak convergence rates for numerical
approximations of SPDEs.
In Section~\ref{section:SPDEsBanach}
below we also briefly review a few well-known results
for Nemytskii
and multiplication operators
in Banach spaces
(see Proposition~\ref{proposition:nonlinearityF},
Proposition~\ref{prop.nonlinearityB},
and Corollary~\ref{corollary:nonlinearityB}
in Section~\ref{section:SPDEsBanach} below)
which provide natural examples
for the possibly nonlinear
test function appearing
in the mild It\^o 
formula in Corollary~\ref{cor:itoauto2}
in Subsection~\ref{sec:mildito}
below.
%
%
%
%
%
%
\subsection{Notation}
\label{sec:notation}
Throughout this article the following notation is
frequently used. 
Let $ \N = \{ 1, 2, 3, \dots \} $ be the set of natural numbers. 
Let $ \N_0 = \N \cup \{ 0 \} $ be the union of $ \{ 0 \} $ and the set of natural numbers. 
For all sets $ A $ and $ B $ let 
$ \mathbb{M}( A, B ) $ be the set of all functions from $ A $ to $ B $. 
For all measurable spaces 
$
  ( \Omega_1,\mathcal{F}_1 ) 
$
and
$
  ( \Omega_2,\mathcal{F}_2 ) 
$
let
$ 
  \mathcal{M}( \mathcal{F}_1, \mathcal{F}_2 )
$ 
be the set of all 
$
  \mathcal{F}_1
$/$
  \mathcal{F}_2 
$-measurable functions.
For all separable $\R$-Hilbert spaces 
$
   (
     \check{H},
     \left< \cdot , \cdot \right>_{ \check{H} },
     \left\| \cdot \right\|_{ \check{H} }
   )
$ 
and 
$
   (
     \hat{H},
     \left< \cdot , \cdot \right>_{ \hat{H} },
     \left\| \cdot \right\|_{ \hat{H} }
   )
$
let
$
  \mathcal{S}( \hat{H}, \check{H} )
$
be the sigma algebra on
$ L( \hat{H}, \check{H} ) $
given by
$
  \mathcal{S}( \hat{H}, \check{H} )
  =
  \sigma_{
    L( \hat{H}, \check{H} )
  }(
    \cup_{ v \in \hat{H} }
    \cup_{ \mathcal{A} \in \mathcal{B}( \check{H} ) }
    \{
      A \in L( \hat{H}, \check{H} )
      \colon
      Av \in \mathcal{A}
    \}
  )
$
(see, e.g.,\ \cite[Section~1.2]{dz92}).
For every $ d \in \N $
and every $ A \in \B(\R^d) $
let
$ \lambda_A \colon \B(A) \to [0,\infty] $
be the Lebesgue-Borel measure on $ A $.
For every set $ X $ 
 let
 $ \#_X \in \N_0 \cup \{\infty \} $
 be the number of elements of $ X $.
  For every measure space
  $ ( \Omega, \mathcal{F}, \nu) $,
  every measurable space $ ( S, \mathcal{S} ) $,
  every set $ R $,
  and every function
  $ f \colon \Omega \to R $
  let
  $ [f]_{\nu, \mathcal{S}} $
  be the set given by
  $ [f]_{\nu, \mathcal{S}} =
  \{
  g \in \mathcal{M}(\mathcal{F},
  \mathcal{S})
  \colon
  (
  \exists\, A \in \mathcal{F} \colon
  \nu(A) = 0 \,\,\text{and}\,\,
  \{\omega \in \Omega \colon
  f(\omega) \neq g(\omega)
  \}
  \subseteq A
  )
  \} $.

\section{Stochastic partial differential equations in Banach spaces}
\label{section:SPDEsBanach}
In this section we recall a few well-known results for SPDEs
on UMD Banach spaces.
In particular,
Proposition~\ref{proposition:nonlinearityF}
below provides natural examples for the possibly
nonlinear test function appearing in the mild It\^o 
formula in Corollary~\ref{cor:itoauto2} 
in Subsection~\ref{sec:mildito} below.
\subsection{Preliminary results}
\label{section:Preliminary}
The following lemma
and its proof can, e.g., 
be found in
Van Neerven~\cite{VanNeerven2010Gamma} (cf.\ \cite[Theorem~6.2]{VanNeerven2010Gamma}
and~\cite[Definition~3.7]{VanNeerven2010Gamma}).
 \begin{lemma}[An ideal property for $ \gamma $-radonifying operators]
 	\label{lemma:multiplication_estimate}
 	Let $ ( U, \langle \cdot, \cdot \rangle_U, \left \| \cdot \right \|_U ) $
 	and
 	$ ( \U, \langle \cdot, \cdot \rangle_\U, \left \| \cdot \right \|_\U ) $
 	be  
 	$ \R $-Hilbert spaces,
 	let $ ( V, \left \| \cdot \right \|_V ) $
 	and
 	$ ( \V, \left \| \cdot \right \|_\V ) $
 	be $ \R $-Banach spaces,
 	and let
 	$ A \in L ( V, \V ) $,
 	$ B \in \gamma ( U, V ) $,
 	$ C \in L ( \U, U ) $.
 	Then it holds that
 	$ A B C \in \gamma ( \U, \V ) $
 	and
 	\begin{equation}
 	\| A B C \|_{\gamma( \U, \V )}
 	\leq
 	\| A \|_{ L ( V, \V ) }
 	\| B \|_{ \gamma ( U, V ) }
 	\| C \|_{ L ( \U, U ) }
 	.
 	\end{equation}
 \end{lemma}
 The next result is an elementary extension of 
 Brze\'zniak et al.\ \cite[Lemma~2.3]{BrzezniakVanNeervenVeraarWeis2008}.
 \begin{lemma}
 	\label{lemma:gamma_estimate}
 	Consider the notation in Subsection~\ref{sec:notation},
 	let $ ( U, \langle \cdot, \cdot \rangle_U, \left \| \cdot \right \|_U ) $
 	be a separable $ \R $-Hilbert space, 
 	let $ ( V, \left \| \cdot \right \|_V ) $
 	and
 	$ ( \V, \left \| \cdot \right \|_\V ) $
 	 be $ \R $-Banach spaces,
 	and let
 	$ \beta \in L^{(2)}(V, \V) $.
 	Then
 	\begin{enumerate}[(i)]
 		\item it holds for all 	$ A_1, A_2 \in \gamma (U, V ) $
 		and all 
 		orthonormal sets
 		$ \mathbb{U} \subseteq U $ of $ U $
 		that
 		there exists a unique $ v \in \V $
 		such that  
 		\begin{equation}
 		\inf_{ \substack{ I \subseteq \mathbb{U}, \\ \#_I < \infty } }
 		\sup_{ \substack{ I \subseteq J \subseteq \mathbb{U}, \\ \#_J < \infty } }
 		\bigg\| v - \smallsum\limits_{ u \in J } \beta( A_1 u, A_2 u ) \bigg\|_\V = 0
 		,
 		\end{equation}
 		\item it holds for all orthonormal bases
 		$ \mathbb{U}_1, \mathbb{U}_2 \subseteq U $
 		of $ U $ that
 		\begin{equation}
 		\sum_{ u \in \mathbb{U}_1 }
 		\beta( A_1 u, A_2 u )
 		=
 		\sum_{ u \in \mathbb{U}_2 }
 		\beta( A_1 u, A_2 u )
 		,
 		\end{equation}
 		\item 
 		it holds for all
 		$ A_1, A_2 \in \gamma (U, V ) $ 
 		and all
 		orthonormal sets $ \mathbb{U} \subseteq U $ of $ U $ that
 		\begin{equation}
 		\bigg \| \smallsum\limits_{u \in \mathbb{U} } \beta( A_1 u, A_2 u ) \bigg \|_\V
 		\leq
 		\| \beta \|_{L^{(2)}(V, \V)}
 		\| A_1 \|_{\gamma(U,V)}
 		\| A_2 \|_{\gamma(U,V)}
 		,
 		\end{equation}
 		and
 		\item it holds for all 
 		orthonormal sets $ \mathbb{U} \subseteq U $ of $ U $
 		that
 		\begin{equation}
 		\begin{split}
 		\bigg( 
 		\gamma(U,V) \times \gamma(U,V) \ni (A_1, A_2)
 		\mapsto
 		\smallsum\limits_{u \in \mathbb{U} }
 		\beta ( A_1 u, A_2 u ) \in \V
 		\bigg)\in L^{(2)}( \gamma(U, V), \V ) 
 		.
 		\end{split}
 		\end{equation}
 	\end{enumerate}
 \end{lemma}

\subsection{Convergence properties of measurable functions}
\begin{lemma}[A characterization for convergence in measure]
	\label{lemma:con_measure}
	Let $ ( \Omega, \mathcal{F}, \nu ) $
	be a finite measure space and let $ R_n \colon \Omega \to \R $,
	$ n \in \N = \{ 1, 2,\ldots \} $, 
	be $ \mathcal{F}/\B(\R) $-measurable
	functions.
	Then the following two statements
	are equivalent:
	\begin{enumerate}[(i)]
		\item It holds that
		\begin{equation}
		\limsup_{n\to \infty}
		\int_\Omega \min \{ 1, |R_n| \} \, d \nu = 0.
		\end{equation}
		\item For every strictly increasing function $ n \colon \N \to \N $
		there exists a strictly increasing function $ m \colon \N \to \N $
		such that
		\begin{equation}
		\nu \bigg( \bigg\{ 
		\omega \in \Omega \colon \limsup_{ k \to \infty }
		| R_{n(m(k))} ( \omega ) | > 0
		\bigg\}\bigg) = 0.
		\end{equation}
		\end{enumerate}
		\end{lemma}
		\begin{lemma}
			\label{lemma:convergence}
			%
			Let $ ( \Omega, \mathcal{F}, \nu ) $
			be a finite measure space,
			let $ (E, d ) $ and $ ( \mathcal{E}, \delta ) $
			be separable pseudometric spaces,
			let $ \phi \colon E \to \mathcal{E} $
			be a continuous function,
			and let
			$ f_n \colon \Omega \to E $, $ n \in \{ 0,1, 2, \ldots \} $,
			be $ \mathcal{F}/\B(E) $-measurable functions
			which
			satisfy
			\begin{equation} 
			\limsup_{n \to \infty} 
			\int_\Omega \min \{ 1, d (f_n, f_0 ) \} \, d\nu = 0 
			.
			\end{equation}
			Then 
			\begin{equation}
			\limsup_{n \to \infty}
			\int_\Omega 
			\min \{ 1, \delta( \phi\circ f_n, \phi \circ f_0 ) \} \, d\nu = 0 
			.
			\end{equation}
			\end{lemma}
			\begin{proof}[Proof of Lemma~\ref{lemma:convergence}]
				Observe that Lemma~\ref{lemma:con_measure}
				and the assumption that
				\begin{equation} 
				\limsup_{n \to \infty} 
				\int_\Omega \min \{ 1, d (f_n, f_0 ) \} \, d\nu = 0 
				\end{equation}
				ensure that for every
				strictly increasing function $ n \colon \N \to \N $
				there exists a strictly increasing function
				$ m \colon \N \to \N $
				such that
				\begin{equation}
				\nu \bigg(
				\bigg\{
				\omega \in \Omega \colon
				\limsup_{ k \to \infty }
				d ( f_{n(m(k))}(\omega) , f_0(\omega ) ) > 0
				\bigg\}
				\bigg)
				=
				0.
				\end{equation}
				The assumption that $ \phi $ is continuous hence shows that
				for every strictly increasing
				function
				$ n \colon \N \to \N $
				there exists
				a strictly increasing
				function $ m \colon \N \to \N $
				such that
				\begin{equation}
				\nu \bigg(
				\bigg\{
				\omega \in \Omega \colon
				\limsup_{ k \to \infty }
				\delta ( \phi ( f_{n(m(k)) } (\omega) ) , \phi( f_0(\omega ) ) ) > 0
				\bigg\}
				\bigg)
				=
				0.
				\end{equation}
				Combining this with
				Lemma~\ref{lemma:con_measure}
				completes the proof of Lemma~\ref{lemma:convergence}.
				\end{proof}
				\begin{corollary}
					\label{corollary:convergence}
					%
					Let $ ( \Omega, \mathcal{F}, \nu ) $
					be a finite measure space,
					let $ (E, d ) $ and $ ( \mathcal{E}, \delta ) $
					be separable pseudometric spaces,
					let $ p, q \in (0,\infty) $,
					let $ \phi \colon E \to \mathcal{E} $
					be a continuous and globally bounded function,
					and let
					$ f_n \colon \Omega \to E $,
					$ n \in \{0, 1, 2, \ldots \} $,
					be
					$ \mathcal{F} / \B(E ) $-measurable functions which
					satisfy
					$ \limsup_{ n \to \infty }
					\int_{\Omega} | d (f_n, f_0) |^p \, d\nu = 0 $.
					Then
					\begin{equation}
					\limsup_{ n \to \infty }
					\int_{\Omega} | \delta ( \phi \circ f_n, \phi \circ f_0 ) |^q \, d\nu = 0.
					\end{equation}
					\end{corollary}
					\begin{proof}[Proof of Corollary~\ref{corollary:convergence}]
						Observe that the assumption that
						$ \limsup_{ n \to \infty }
						\int_{\Omega} | d (f_n, f_0) |^p \, d\nu = 0 $
						and H\"older's inequality
						ensure that
						$ \limsup_{ n \to \infty }
						\int_{\Omega} | d (f_n, f_0) |^{ \min \{ p, 1 \} } \, d\nu = 0 $.
						Hence, we obtain that
						\begin{equation}
						\limsup_{ n \to \infty }
						\int_{\Omega}
						\min \{ 1, d(f_n, f_0 ) \} \, d\nu
						\leq
						\limsup_{ n \to \infty }
						\int_{\Omega}
						| \min \{ 1, d(f_n, f_0) \} |^{ \min \{p,1 \}} \, d\nu
						=
						0.
						\end{equation}
						This allows us to apply Lemma~\ref{lemma:convergence}
						to obtain that
						\begin{equation}
						\limsup_{ n \to \infty }
						\int_{\Omega}
						\min \{ 1, \delta ( \phi \circ f_n, \phi \circ f_0 ) \} \, d \nu = 0.
						\end{equation}
						The fact that the function
						$ [0,\infty) \ni x \mapsto | x |^q \in [0,\infty) $
						is continuous and again Lemma~\ref{lemma:convergence} 
						hence show that
						\begin{equation}
						\limsup_{ n \to \infty }
						\int_{\Omega}
						\min \{ 1, | \delta( \phi \circ f_n, \phi \circ f_0 ) |^q \} \, d \nu = 0.
						\end{equation}
						Combining this and, e.g., 
						Klenke~\cite[Corollary~6.26]{Klenke2008}
						with the fact that
						$ \sup( \{ | \delta( ( \phi \circ f_n ) ( \omega ),
						( \phi \circ f_0 )( \omega ) ) |^p
						\colon 
						\omega \in \Omega, n \in \N 
						\}
						\cup \{ 0 \}
						) < \infty $
						ensures that
						\begin{equation}
						\limsup_{ n \to \infty }
						\int_{\Omega}
						| \delta( \phi \circ f_n, \phi \circ f_0 ) |^q \, d\nu = 0.
						\end{equation}
						The proof of Corollary~\ref{corollary:convergence}
						it thus completed.
						\end{proof}
\subsection{Regular test functions}
\label{subsection:RegularTestFunctions}
\begin{proposition}
	\label{proposition:nonlinearityF}
	Consider the notation in Subsection~\ref{sec:notation},
	let $ k, l, d, n \in \N $,
	$ p \in [1, \infty) $,
	$ q \in ( n p, \infty) $,
	let
	$ \O \in \mathcal{B}( \R^d ) $
	be a bounded set,
	let $ f \colon \R^k \to \R^l $
	be an $ n $-times continuously differentiable
	function with globally bounded  
	derivatives, 
	and
	let $ F \colon L^q(\lambda_\O; \R^k ) \to L^p(\lambda_\O; \R^l ) $
	be the function which satisfies for all $ v \in \L^q(\lambda_\O; \R^k ) $ that
	\begin{equation}
	F ( [ v ]_{\lambda_\O, \B(\R^l)} ) 
	=
	[ \{ f ( v ( x ) ) \}_{x \in \O} ]_{ \lambda_\O, \B(\R^l) }
	=
	[ f \circ v ]_{ \lambda_\O, \B(\R^l) }
	. 
	\end{equation}
	Then
	\begin{enumerate}[(i)]
		\item \label{item:1} it holds that $ F $ is $ n $-times continuously Fr\'echet differentiable
		with globally bounded derivatives,
		\item \label{item:2} it holds for all
		$ m \in \{ 1, 2, \ldots, n \} $,
		$ v, u_1, \ldots, u_m \in \L^q(\lambda_\O; \R^k ) $
		that
		\begin{equation}
		\begin{split}
		&
		F^{( m )}( [ v ]_{\lambda_\O, \B(\R^k ) } )
		( [ u_1 ]_{\lambda_\O, \B(\R^k ) }, \ldots, 
		[ u_m ]_{\lambda_\O, \B(\R^k ) } )
		\\
		&
		=
		[ \{ f^{(m)}(v(x)) ( u_1(x), \ldots, u_m(x) ) \}_{ x \in \O } ]_{\lambda_\O, \B(\R^l)}
		,
		\end{split}
		\end{equation}
		\item \label{item:3}  it holds for all
		$ m \in \{ 1, 2, \ldots, n \} $, $ r \in [mp, \infty) $ 
		that
		\begin{equation}
		\begin{split}
		&
		\sup_{ v \in L^q( \lambda_\O; \R^k ) }
		\sup_{
			u_1, \ldots, u_m \in L^{ \max \{r, q \} }( \lambda_\O; \R^k ) \backslash \{ 0 \}
		}
		\Bigg[
		\frac{
			\| F^{(m)}(v) (u_1, \ldots, u_m ) \|_{L^p(\lambda_\O; \R^l ) 
			}}{
			\| u_1 \|_{L^r(\lambda_\O; \R^k )}
			\cdot\ldots\cdot
			\| u_m \|_{L^r(\lambda_\O; \R^k )}
		}
		\Bigg]
		\\
		&
		\leq
		\bigg[
		\sup_{ x \in \R^k }
		\| f^{(m)}(x) \|_{L^{(m)}(\R^k, \R^l )}
		\bigg]
		[ \lambda_{\R^d}( \O ) ]^{[ \frac{1}{p} - \frac{m}{r} ]}
		< \infty
		,
		\end{split}
		\end{equation}
		\item \label{item:4} it holds for all $ m \in \{ 1, 2, \ldots, n \} $,
		$ r, s \in (p,\infty) $, 
		$ v, w \in L^{\max\{r,q\}}( \lambda_\O; \R^k ) $
		with 
		$ \frac{1}{r}+ \frac{m}{s} \leq \frac{1}{p} $
		that
		\begin{equation}
		\begin{split}
		& 
		\sup_{  
			u_1, \ldots, u_m \in L^{ \max \{s, q\}}( \lambda_\O; \R^k )\backslash \{0\} 
		}
		\Bigg[
		\frac{
			\| ( F^{(m)}(v) - F^{(m)}(w) ) ( u_1, \ldots, u_m ) \|_{   L^p(\lambda_\O; \R^l )  } 
		}
		{
			\| u_1 \|_{ L^{ s  } (\lambda_\O; \R^k) }
			\cdot
			\ldots
			\cdot
			\| u_m \|_{ L^{ s  } (\lambda_\O; \R^k) }	
		}
		\Bigg]
		\\
		&
		\leq
		\Bigg[
		\sup_{\substack{ x,y \in \R^k, \\ x \neq y }}
		\frac{
			\| f^{(m)}(x) - f^{(m)}(y) \|_{L^{(m)}(\R^k, \R^l)}
		}{
		\| x - y \|_{\R^k}
	}
	\Bigg]
	[ \lambda_{\R^d}( \O ) ]^{ [ \frac{1}{p} - \frac{1}{r} - \frac{m}{s} ]}
	\|  v - w \|_{ L^r( \lambda_\O; \R^k) } 
	%
	%
	,
	\end{split}
	\end{equation}
	and
	\item \label{item:5} 
	it holds for all $ m \in \{ 1, 2, \ldots, n \} $,
	$ r \in [(m+1)p,\infty) $, 
	$ v, w \in L^{\max\{r, q\}}( \lambda_\O; \R^k ) $
	that
	\begin{equation}
	\begin{split}
	& 
	\sup_{  
		u_1, \ldots, u_m \in L^{\max\{r,q\}}( \lambda_\O; \R^k )\backslash \{0\} 
	}
	\Bigg[
	\frac{
		\| ( F^{(m)}(v) - F^{(m)}(w) ) ( u_1, \ldots, u_m ) \|_{   L^p(\lambda_\O; \R^l )  } 
	}
	{
		\| u_1 \|_{ L^{ r } (\lambda_\O; \R^k) }
		\cdot
		\ldots
		\cdot
		\| u_m \|_{ L^{ r } (\lambda_\O; \R^k) }	
	}
	\Bigg]
	\\
	&
	\leq
	\Bigg[
	\sup_{\substack{ x,y \in \R^k, \\ x\neq y }}
	\frac{
		\| f^{(m)}(x) - f^{(m)}(y) \|_{L^{(m)}(\R^k, \R^l)}
	}{
	\| x - y \|_{\R^k}
}
\Bigg]
[ \lambda_{\R^d}( \O ) ]^{ [ \frac{1}{p} - \frac{m+1}{r} ]}
\|  v - w \|_{ L^r( \lambda_\O; \R^k) }  
.
\end{split}
\end{equation}
\end{enumerate}
\end{proposition}
\begin{proof}[Proof of Proposition~\ref{proposition:nonlinearityF}]
	Throughout this proof we
	assume w.l.o.g.\ that $ \lambda_{\R^d}( \O ) > 0 $.
	We claim that for all $ m \in \{ 1, 2, \ldots, n \} $ it holds  
	\begin{enumerate}[(a)]
		\item \label{item:a} that $ F $ is $ m $-times 
		Fr\'echet differentiable 
		and
		\item \label{item:b} 
		that for all
		$ v, u_1, \ldots, u_m \in \L^q(\lambda_\O; \R^k ) $
		it holds that
		\begin{equation}
		\begin{split}
		&
		F^{( m )}( [ v ]_{\lambda_\O, \B(\R^k ) } )
		( [ u_1 ]_{\lambda_\O, \B(\R^k ) }, \ldots, 
		[ u_m ]_{\lambda_\O, \B(\R^k ) } )
		\\
		&
		=
		[ \{ f^{(m)}(v(x)) ( u_1(x), \ldots, u_m(x) ) \}_{ x \in \O } ]_{\lambda_\O, \B(\R^l)}
		.
		\end{split}
		\end{equation}
	\end{enumerate}
	We now prove item~\eqref{item:a}
	and item~\eqref{item:b}
	by induction on
	$ m \in \{ 1, 2, \ldots, n \} $.
	For the base case $ m = 1 $
	we note that 
	Minkowski's integral inequality
	and
	H\"older's inequality show that
	for all   
	$ v, h \in \L^q( \lambda_\O; \R^k ) $,
	$ \varepsilon \in (0,\infty) $
	it holds that
	\begin{equation}
	\begin{split}
	\label{eq:first_estimate}
	&
	\| 
	f \circ ( v + h ) - f \circ  v - ( f' \circ v ) h
	\|_{ \L^p( \lambda_\O; \R^l )  }
	\\
	&
	\leq 
	\int_0^1 
	\| 
	[ f' \circ ( v + r h ) - f' \circ v ] 
	h
	\|_{ \L^p( \lambda_\O; \R^l) }
	\, dr  
	\\
	&
	\leq
	\int_0^1 
	\| f' \circ ( v + r h ) - f' \circ v \|_{\L^{ p( 1 + \nicefrac{1}{\varepsilon} ) }( \lambda_\O; L( \R^k, \R^l ) ) } \, dr 
	\,
	\| h \|_{\L^{ p( 1 + \varepsilon ) }( \lambda_\O; \R^k ) }
	.
	\end{split}
	\end{equation}
	Next observe that
	Corollary~\ref{corollary:convergence}
	(with
	$ ( \Omega, \mathcal{F}, \nu ) 
	=
	( \O, \B( \O ), \lambda_\O ) $,
	$ E = \R^k $,
	$ \mathcal{E} = L ( \R^k, \R^l ) $,
	$ p = p(1+\varepsilon) $,
	$ q = p ( 1 + \nicefrac{1}{\varepsilon} ) $,
	$ \phi = f' $,
	$ f_0 = v $,
	$ f_j = v + r h_j $
	for $ r \in [0,1] $,
	$ j \in \N $,
	$ v \in \L^q( \lambda_\O; \R^k ) $,
	$ ( h_j )_{ j \in \N } \in \{
	( u_j )_{ j \in \N } \subseteq \L^{p(1+\varepsilon)} ( \lambda_\O; \R^k ) 
	\colon
	\limsup_{j \to \infty }
	\| u_j \|_{\L^{p(1+\varepsilon)}( \lambda_\O; \R^k )}
	= 0
	\} $,
	$ \varepsilon \in (0,\infty ) $
	in the notation of Corollary~\ref{corollary:convergence}),
	the fact that
	$ \sup_{ x \in \R^k } \| f'(x) \|_{L(\R^k, \R^l )} < \infty $,
	and
	the fact that $ f' $ is continuous
	ensure that for all
	$ r \in [0,1] $,
	$ v \in \L^q( \lambda_\O; \R^k ) $,
	$ \varepsilon \in (0, \infty) $,
	$ (h_j)_{ j \in \N} \subseteq \L^{p ( 1 + \varepsilon ) }( \lambda_\O; \R^k ) $
	with
	$ \limsup_{ j \to \infty }
	\| h_j \|_{\L^{p(1+\varepsilon)}( \lambda_\O; \R^k ) } = 0 $
	it holds that
	\begin{equation}
	\limsup_{ j \to \infty }
	\| f' \circ ( v + r h_j ) - f' \circ v \|_{\L^{ p(1+\nicefrac{1}{\varepsilon}) }( \lambda_\O; L( \R^k, \R^l ) ) } = 0 .
	\end{equation}
	This, the fact that
	$ \sup_{ x \in \R^k } \| f'(x) \|_{L(\R^k, \R^l)} < \infty $,
	and Lebesgue's theorem of dominated convergence
	prove that for all
	$ v \in \L^q( \lambda_\O; \R^k ) $,
	$ \varepsilon \in (0,\infty) $,
	$ ( h_j )_{ j \in \N } \subseteq \L^{p(1+\varepsilon)}( \lambda_\O; \R^k ) $
	with
	$ \limsup_{ j \to \infty }
	\| h_j \|_{\L^{p(1+\varepsilon)}( \lambda_\O; \R^k ) } = 0 $
	it holds
	that
	\begin{equation}
	\limsup_{ j \to \infty }
	\bigg(
	\int_0^1
	\| f' \circ ( v + r h_j ) - f' \circ v \|_{\L^{ p(1+\nicefrac{1}{\varepsilon}) }( \lambda_\O; L( \R^k, \R^l ) ) } 
	\, dr
	\bigg)
	=
	0.
	\end{equation}
	This together with
	H\"older's inequality and~\eqref{eq:first_estimate}
	implies that for all 
	$ v \in \L^q( \lambda_\O; \R^k ) $,
	$ \varepsilon \in (0,\nicefrac{q}{p}-1) $,
	$ ( h_j )_{j \in \N} \subseteq \L^{p(1+\varepsilon)}(\lambda_\O; \R^k) $
	with
	$ \limsup_{ j \to \infty} \| h_j \|_{\L^{p(1+\varepsilon)}( \lambda_\O; \R^k )} = 0 $
	and
	$ \forall \, j \in \N \colon 
	\| h_j \|_{\L^{p(1+\varepsilon)}( \lambda_\O; \R^k )} > 0 $
	it holds that
	\begin{equation}
	\begin{split}
	&
	\limsup_{ j \to \infty} 
	\bigg(
	\frac{ \| 
		f \circ ( v + h_j ) - f \circ v - ( f' \circ v )  h_j 
		\|_{\L^p( \lambda_\O; \R^l ) } }
	{ \| h_j \|_{ \L^{ p( 1 + \varepsilon ) }  ( \lambda_\O; \R^k ) } }
	\bigg)
	%
	%
	\\
	&
	\leq
	\limsup_{ j \to \infty }
	\bigg(
	\int_0^1
	\| f' \circ ( v + r h_j ) - f' \circ v \|_{\L^{ p(1+\nicefrac{1}{\varepsilon}) }( \lambda_\O; L( \R^k, \R^l ) ) } 
	\, dr
	\bigg)
	=
	0
	.
	\end{split}
	\end{equation}
	H\"older's inequality hence shows that for all
	$ v \in \L^q( \lambda_\O; \R^k ) $,
	$ \varepsilon \in (0, \nicefrac{q}{p} - 1 ) $,
	$ ( h_j )_{j \in \N} \subseteq \L^{p(1+\varepsilon)}( \lambda_\O; \R^k ) $
	with 
	$ \limsup_{ j \to \infty } \| h_j \|_{\L^q( \lambda_\O; \R^k ) } = 0 $
	and
	$ \forall \, j \in \N \colon \| h_j \|_{\L^q( \lambda_\O; \R^k ) } > 0 $
	it holds that
	\begin{equation}
	\begin{split}
	&
	\limsup_{ j \to \infty} 
	\bigg(
	\frac{ \| 
		f \circ ( v + h_j ) - f \circ v - ( f' \circ v )  h_j
		\|_{\L^p( \lambda_\O;  \R^l  ) } }
	{ \| h_j \|_{ \L^{ q }  ( \lambda_\O; \R^k ) } }
	\bigg) 
	=
	0
	.
	\end{split}
	\end{equation}
	This demonstrates that 
	$ F $ is Fr\'echet differentiable and that
	for all 
	$ v, h \in \L^q( \lambda_\O; \R^k ) $
	it holds that
	\begin{equation}
	\label{eq:base1}
	F'( [ v ]_{\lambda_\O, \B ( \R^k ) } ) [ h ]_{\lambda_\O, \B( \R^k )} 
	= 
	[ \{ f'(  v(x) ) h(x) \}_{ x \in \O } ]_{\lambda_\O, \B( \R^l ) }
	.
	\end{equation} 
	This proves item~\eqref{item:a} and item~\eqref{item:b}
	in the base case $ m = 1 $.
	For the induction step
	$ \N \cap [0, n-1] \ni m \to m+1 \in \{ 1, 2, \ldots, n \} $
	assume that there exists a natural number
	$ m \in \N \cap [0, n-1] $
	such that item~\eqref{item:a}
	and item~\eqref{item:b}
	hold for $ m = m $.
	Next observe that
	Minkowski's integral inequality
	and
	H\"older's inequality show that
	for all  
	$ v, h, u_1 \ldots, u_m \in \L^q( \lambda_\O; \R^k ) $, 
	$ \varepsilon \in ( 0, \frac{q}{p(1+m) } - 1  ) $
	it holds that
	\begin{equation}
	\begin{split}
	\label{eq:first_estimate2}
	&
	\| 
	[ ( f^{(m)} \circ ( v + h ) )  
	- 
	( f^{(m)} \circ v )]
	( u_1, \ldots, u_m ) - ( f^{(m+1)} \circ v ) 
	( h, u_1, \ldots, u_m ) 
	\|_{ \L^p( \lambda_\O; \R^l ) }
	%
	\\
	&
	\leq 
	\int_0^1 
	\| 
	[ f^{(m+1)} \circ ( v + r h ) - f^{(m+1)} \circ v ] 
	( h, u_1, \ldots, u_m )
	\|_{ \L^p( \lambda_\O; \R^l ) }
	\, dr  
	\\
	&
	\leq
	\int_0^1 
	\| f^{(m+1)} \circ ( v + r h ) - f^{(m+1)} \circ v \|_{\L^{ p( 1 + \nicefrac{1}{\varepsilon} ) }( \lambda_\O; L^{(m+1)} ( \R^k, \R^l ) ) } \, dr 
	\\
	&
	\quad 
	\cdot
	\| h \|_{\L^{ p( 1 + \varepsilon )( 1 + m ) }( \lambda_\O; \R^k ) }
	\prod_{i=1}^{m}
	\| u_i \|_{\L^{ p( 1 + \varepsilon )( 1 + m ) }( \lambda_\O; \R^k ) }
	.
	\end{split}
	\end{equation}
	Moreover, note that
	Corollary~\ref{corollary:convergence}
	(with
	$ ( \Omega, \mathcal{F}, \nu ) 
	=
	( \O, \B( \O ), \lambda_\O ) $,
	$ E = \R^k $,
	$ \mathcal{E} = L^{(m+1)} ( \R^k, \R^l ) $,
	$ p = p(1+\varepsilon)(1+m) $,
	$ q = p ( 1 + \nicefrac{1}{\varepsilon} ) $,
	$ \phi = f^{(m+1)} $,
	$ f_0 = v $,
	$ f_j = v + r h_j $
	for $ r \in [0,1] $,
	$ j \in \N $,
	$ v \in \L^q( \lambda_\O; \R^k ) $,
	$ ( h_j )_{ j \in \N } \in \{
	( u_j )_{ j \in \N } \subseteq \L^{p(1+\varepsilon)(1+m)} ( \lambda_\O; \R^k ) 
	\colon
	$
	$
	\limsup_{j \to \infty }
	\| u_j \|_{\L^{p(1+\varepsilon)(1+m)}( \lambda_\O; \R^k )}
	= 0
	\} $,
	$ \varepsilon \in (0,\infty ) $
	in the notation of Corollary~\ref{corollary:convergence}),
	the fact that
	$ \sup_{ x \in \R^k } \| f^{(m+1)}( x ) \|_{L^{(m+1)}(\R^k, \R^l )}
	< \infty $,
	and
	the fact that $ f^{(m+1)} $ is continuous 
	ensure that for all
	$ r \in [0,1] $,
	$ v \in \L^q( \lambda_\O; \R^k ) $,
	$ \varepsilon \in (0,\infty) $,
	$ (h_j)_{j \in \N} \subseteq \L^{p(1+\varepsilon)(1+m)}(\lambda_\O; \R^k ) $
	with
	$ \limsup_{ j \to \infty } \| h_j \|_{\L^{ p ( 1 + \varepsilon )(1+m) }( \lambda_\O; \R^k) } = 0 $
	it holds that
	\begin{equation}
	\begin{split}
	&
	\limsup_{j \to \infty}
	\| f^{(m+1)} \circ ( v + r h_j ) - f^{(m+1)} \circ v    
	\|_{ \L^{p( 1 + \nicefrac{1}{\varepsilon})}( \lambda_\O; L^{(m+1)}(\R^k, \R^l) ) } %
	=
	0.
	\end{split}
	\end{equation}
	This, the fact that
	$ \sup_{ x \in \R^k } \| f^{(m+1)}(x) \|_{L^{(m+1)}(\R^k, \R^l)}< \infty $,
	and Lebesgue's theorem of dominated convergence
	prove that for all
	$ v \in \L^q( \lambda_\O; \R^k ) $, $ \varepsilon \in (0,\infty) $,
	$ ( h_j )_{ j \in \N} 
	\subseteq \L^{p(1+\varepsilon)(1+m)}(\lambda_\O; \R^k ) $
	with
	$ \limsup_{ j \to \infty} \| h_j \|_{\L^{p(1+\varepsilon)(1+m)}( \lambda_\O; \R^k )}= 0 $
	it holds that 
	\begin{equation}
	\begin{split}
	&
	\limsup_{j \to \infty}
	\bigg(
	\int_0^1 
	\|
	f^{(m+1)} \circ ( v + r h_j ) - f^{(m+1)} \circ v ] 
	\|_{ \L^{ p( 1 + \nicefrac{1}{\varepsilon} ) } ( \lambda_\O ; L^{(m+1)}(\R^k, \R^l) ) }
	\, dr  
	\bigg) 
	=
	0. 
	\end{split}
	\end{equation}
	The fact that
	$ \forall \, \varepsilon \in (0, \frac{q}{p( 1 + m )} - 1 ) \colon
	\L^{ p (1+\varepsilon)(1+m) } ( \lambda_\O; \R^k )
	\subseteq 
	\L^q( \lambda_\O; \R^k ) $
	and~\eqref{eq:first_estimate2}
	hence
	imply that for all 
	$ v \in \L^q( \lambda_\O; \R^k ) $,
	$ \varepsilon \in (0,\frac{q}{p(1+m)}-1) $,
	$ ( h_j )_{j \in \N} \subseteq \L^{p(1+\varepsilon)(1+m)}(\lambda_\O; \R^k) $
	with
	$ \limsup_{ j \to \infty} \| h_j \|_{\L^{p(1+\varepsilon)(1+m)}( \lambda_\O; \R^k )} = 0 $
	and
	$ \forall \, j \in \N \colon
	\| h_j \|_{\L^{ p( 1 + \varepsilon )( 1 + m ) }( \lambda_\O; \R^k ) }
	> 0 $ 
	it holds that
	\begin{equation}
	\begin{split}
	&
	\limsup_{ j \to \infty} 
	\sup_{ \substack{ u_1, \ldots, u_m \in \L^q( \lambda_\O; \R^k ), 
			\\
			\prod^m_{i=1}
			\| u_i \|_{\L^{ p ( 1 + \varepsilon) ( 1 + m ) }( \lambda_\O; \R^k ) }	> 0
		} }
		\!\!\!
		\tfrac{ \| 
			[( f^{(m)} \circ ( v + h_j ) )  
			- 
			( f^{(m)} \circ v )
			]
			( u_1, \ldots, u_m )
			- 
			( f^{(m+1)} \circ v ) 
			( h_j, u_1, \ldots, u_m )
			\|_{\L^p( \lambda_\O; \R^l ) } }
		{ 
			\| h_j \|_{\L^{ p( 1 + \varepsilon )( 1 + m ) }( \lambda_\O; \R^k ) }
			\prod_{i=1}^m \| u_i \|_{ \L^{ p( 1 + \varepsilon ) (1+m)}  ( \lambda_\O; \R^k ) } }
		\\
		&
		\leq
		\limsup_{j \to \infty}
		\bigg(
		\int_0^1 
		\| f^{(m+1)} \circ ( v + r h_j  ) - f^{(m+1)} \circ v 
		\|_{ \L^{ p( 1 + \nicefrac{1}{\varepsilon} ) } ( \lambda_\O ; L^{(m+1)}( \R^k, \R^l) ) }
		\, dr  
		\bigg) 
		=
		0
		.
		\end{split}
		\end{equation}
		%
		%
		%
		%
		%
		H\"older's inequality therefore shows that for all 
		$ v \in \L^q( \lambda_\O; \R^k ) $
		and all
		$ ( h_j )_{ j \in \N } \subseteq \L^q( \lambda_\O; \R^k ) $
		with
		$ \limsup_{j \to \infty} 
		\| h_j \|_{ \L^q( \lambda_\O; \R^k ) } = 0 $
		and
		$ \forall \, j \in \N \colon \| h_j \|_{\L^q( \lambda_\O; \R^k ) } > 0 $
		it holds that
		\begin{multline}
		\limsup_{ j \to \infty} 
		\sup_{ \substack{ u_1, \ldots, u_m \in \L^q( \lambda_\O; \R^k ), 
				\\
				\prod^m_{i=1}
				\| u_i \|_{\L^q( \lambda_\O; \R^k ) }	> 0
			} }
			\!\!\!
			\bigg(
			\tfrac{ \| 
				[( f^{(m)} \circ ( v + h_j ) )  
				- 
				( f^{(m)} \circ v )
				]
				( u_1, \ldots, u_m )
				- 
				( f^{(m+1)} \circ v ) 
				( h_j, u_1, \ldots, u_m )
				\|_{\L^p( \lambda_\O; \R^l ) } }
			{ 
				\| h_j \|_{\L^{ q }( \lambda_\O; \R^k ) }
				\prod_{i=1}^m \| u_i \|_{ \L^{ q }  ( \lambda_\O; \R^k ) } }
			\bigg)
			\\=
			0
			.
			\end{multline}
			The induction hypothesis hence
			implies that 
			$ F^{(m)} $ is Fr\'echet differentiable and that for all 
			$ v, h, u_1, \ldots, u_m \in \L^q( \lambda_\O; \R^k ) $
			it holds that
			\begin{equation}
			\begin{split}
			\label{eq:F_representation}
			&
			F^{(m+1)}( [ v ]_{\lambda_\O, \B( \R^k ) } )( [ h ]_{\lambda_\O, \B( \R^k ) }, 
			[ u_1 ]_{\lambda_\O, \B( \R^k ) }, \ldots, 
			[ u_m ]_{\lambda_\O, \B( \R^k ) } ) 
			\\
			&
			=
			[
			\{ f^{(m+1)}( v(x) ) ( h(x) , u_1(x), \ldots u_m(x) ) \}_{ x \in \O } 
			]_{ \lambda_\O, \B( \R^l ) }
			.
			\end{split}
			\end{equation}
			This establishes item~\eqref{item:a} and item~\eqref{item:b}
			in the case $ m + 1 $.
			Induction thus completes the proof of item~\eqref{item:a}
			and item~\eqref{item:b}.

			In the next step we observe that 
			H\"older's inequality ensures that for all
			$ m \in \{ 1, 2, \ldots, n \} $,
			$ v, w \in \L^q( \lambda_\O; \R^k ) $,
			$ r, s \in ( p, \infty ) $,
			$ u_1, \ldots, u_m \in \L^s( \lambda_\O; \R^k ) $
			with
			$ \frac{1}{r} + \frac{m}{s} \leq \frac{1}{p} $
			it holds that
			\begin{equation} 
			\begin{split}
			\label{eq:F_cont} 
			&
			\| [f^{(m)} \circ v  - f^{(m)} \circ w ]
			( u_1, \ldots,  u_m )  \|_{ \L^p( \lambda_\O; \R^l ) }
			\\
			&
			\leq
			[ \lambda_{\R^d}( \O ) ]^{ [ \frac{1}{p} - \frac{1}{r}- \frac{m}{s} ] }
			\| f^{(m)}\circ v - f^{(m)} \circ w \|_{ \L^r( \lambda_\O; L^{(m)}( \R^k, \R^l) ) }
			\prod_{i=1}^{m}
			\| u_i \|_{ \L^{ s } (\lambda_\O; \R^k) }
			.
			\end{split}
			\end{equation}
			This implies that for all
			$ m \in \{ 1, 2, \ldots, n \} $,
			$ v, w \in \L^q( \lambda_\O; \R^k ) $,
			$ r, s \in ( p, \infty ) $
			with $ \frac{1}{r} + \frac{m}{s} \leq \frac{1}{p} $
			it holds that
			\begin{equation} 
			\begin{split}
			\label{eq:F_cont2} 
			&
			\sup_{ \substack{ u_1, \ldots, u_m \in \L^s( \lambda_\O; \R^k ), 
					\\
					\prod^m_{i=1}
					\| u_i \|_{\L^s ( \lambda_\O; \R^k ) }	> 0
				} }
				\Bigg( 
				\frac{
					\| [f^{(m)} \circ v  - f^{(m)} \circ w ]
					( u_1, \ldots,  u_m )  \|_{ \L^p( \lambda_\O; \R^l ) }
				}
				{
					\| u_1 \|_{\L^s ( \lambda_\O; \R^k ) }
					\cdot
					\ldots
					\cdot
					\| u_m \|_{\L^s ( \lambda_\O; \R^k ) }
				}
				\Bigg)
				\\
				&
				\leq
				[ \lambda_{\R^d}( \O ) ]^{ [ \frac{1}{p} - \frac{1}{r}- \frac{m}{s} ] }
				\| f^{(m)}\circ v - f^{(m)} \circ w \|_{ \L^r( \lambda_\O; L^{(m)}( \R^k, \R^l) ) }
				.
				\end{split}
				\end{equation}
				Corollary~\ref{corollary:convergence}
				(with
				$ ( \Omega, \mathcal{F}, \nu ) = ( \O, \B(\O), \lambda_\O ) $,
				$ E = \R^k $,
				$ \mathcal{E} = L^{(n)}( \R^k, \R^l ) $,
				$ p = q $,
				$ q = r $,
				$ \phi = f^{(n)} $,
				$ f_j = v_j $
				for 
				$ r \in (0,\infty) $,
				$ j \in \N_0 $
				in the notation of 
				Corollary~\ref{corollary:convergence})
				and the fact that
				$ \sup_{ x \in \R^k } \| f^{(n)} ( x ) \|_{L^{(n)}( \R^k, \R^l ) } < \infty $
				hence show that for all
				$ (v_j)_{ j \in \N_0 } \subseteq \L^q( \lambda_\O; \R^k ) $,
				$ r, s \in (p,\infty) $
				with 
				$ \limsup_{ j \to \infty } \| v_j \|_{\L^q( \lambda_\O; \R^k )} = 0 $
				and
				$ \frac{1}{r} + \frac{n}{s} \leq \frac{1}{p} $
				it holds that
				\begin{equation} 
				\begin{split}
				\label{eq:F_cont3} 
				&
				\limsup_{ j \to \infty }
				\sup_{ \substack{ u_1, \ldots, u_n \in \L^s( \lambda_\O; \R^k ), 
						\\
						\prod^m_{i=1}
						\| u_i \|_{\L^s ( \lambda_\O; \R^k ) }	> 0
					} }
					\Bigg( 
					\frac{
						\| [f^{(n)} \circ v_j  - f^{(n)} \circ v_0 ]
						( u_1, \ldots,  u_n )  \|_{ \L^p( \lambda_\O; \R^l ) }
					}
					{
						\| u_1 \|_{\L^s ( \lambda_\O; \R^k ) }
						\cdot
						\ldots
						\cdot
						\| u_n \|_{\L^s ( \lambda_\O; \R^k ) }
					}
					\Bigg)
					\\
					&
					\leq
					[ \lambda_{\R^d}( \O ) ]^{ [ \frac{1}{p} - \frac{1}{r}- \frac{m}{s} ] }
					\bigg[
					\limsup_{ j \to \infty }
					\| f^{(n)}\circ v_j - f^{(n)} \circ v_0 \|_{ \L^r( \lambda_\O; L^{(n)}( \R^k, \R^l) ) } 
					\bigg]
					=
					0
					.
					\end{split}
					\end{equation}
					This establishes that
					$ F^{(n)} $ is continuous.
					Combining this with item~\eqref{item:a} and item~\eqref{item:b} proves
					item~\eqref{item:1} and item~\eqref{item:2}.
					Next note that H\"older's inequality shows that for all
					$ m \in \{ 1, 2, \ldots, n \} $,
					$ r \in [mp, \infty) $,
					$ v \in \L^q( \lambda_\O; \R^k ) $,
					$ u_1, \ldots, u_m \in \L^{ \max \{ r, q \} }( \lambda_\O; \R^k ) $
					it holds that
					\begin{equation}
					\begin{split}
					&
					\| ( f^{(m)} \circ v ) ( u_1, \ldots, u_m ) \|_{\L^p( \lambda_\O; \R^l ) }
					\\
					&
					\leq
					\bigg[
					\sup_{ x \in \R^k }
					\| f^{(m)}(x) \|_{L^{(m)}(\R^k, \R^l )}
					\bigg]
					[ \lambda_{\R^d}( \O)]^{[ \frac{1}{p} - \frac{m}{r} ]}
					\prod_{i=1}^m
					\| u_i \|_{\L^r( \lambda_\O; \R^k ) }
					.
					\end{split}
					\end{equation}
					This and item~\eqref{item:2}
					imply that for all $ m \in \{ 1, 2 , \ldots, n \} $,
					$ r \in [mp, \infty) $,
					$ v \in \L^q( \lambda_\O; \R^k ) $ it holds that
					\begin{equation}
					\begin{split}
					&
					\sup_{ \substack{
							u_1, \ldots, u_m \in \L^{\max\{ r, q \} }( \lambda_\O; \R^k ),
							\\
							\prod^m_{i=1}
							\| u_i \|_{\L^r(\lambda_\O; \R^k ) }
							> 0
						}}
						\Bigg(
						\frac{
							\| F^{(m)}( [ v ]_{\lambda_\O, \B(\R^k ) } ) 
							( 
							[ u_1 ]_{\lambda_\O, \B(\R^k ) },
							\ldots, 
							[ u_m ]_{\lambda_\O, \B(\R^k ) }
							)
							\|_{ L^p ( \lambda_\O; \R^l ) }
						}{
						\| u_1 \|_{\L^r(\lambda_\O; \R^k )}
						\cdot
						\ldots
						\cdot
						\| u_m \|_{\L^r(\lambda_\O; \R^k )}
					}
					\Bigg)
					\\
					&
					\leq
					\bigg[\sup_{x\in \R^k} \| f^{(m)}(x) \|_{L^{(m)}(\R^k, \R^l )}\bigg]
					[ \lambda_{\R^d}( \O ) ]^{[ \frac{1}{p} - \frac{m}{r} ]}
					.
					\end{split}
					\end{equation}
					Hence, we obtain that for all
					$ m \in \{ 1, 2, \ldots, n \} $,
					$ r \in [mp, \infty) $
					it holds that
					\begin{equation}
					\begin{split}
					&
					\sup_{ v \in L^q( \lambda_\O; \R^k ) }
					\sup_{
						u_1, \ldots, u_m \in L^{ \max \{r, q \} }( \lambda_\O; \R^k ) \backslash \{ 0 \}
					}
					\Bigg(
					\frac{
						\| F^{(m)}(v) (u_1, \ldots, u_m ) \|_{L^p(\lambda_\O; \R^l ) 
						}}{
						\| u_1 \|_{L^r(\lambda_\O; \R^k )}
						\cdot
						\ldots
						\cdot
						\| u_m \|_{L^r(\lambda_\O; \R^k )}
					}
					\Bigg)
					\\
					&
					\leq
					\bigg[
					\sup_{ x \in \R^k }
					\| f^{(m)}(x) \|_{L^{(m)}(\R^k, \R^l )}
					\bigg]
					[ \lambda_{\R^d}( \O ) ]^{[ \frac{1}{p} - \frac{m}{r} ]}
					< \infty
					.
					\end{split}
					\end{equation}
					This proves item~\eqref{item:3}.
					In the next step we observe that~\eqref{eq:F_cont2} 
					assures that for all
					$ m \in \{ 1, 2, \ldots, n \} $,
					$ r, s \in (p, \infty) $,
					$ v, w \in \L^{ \max \{ r, q \} }( \lambda_\O; \R^k ) $
					with
					$ \frac{1}{r} + \frac{m}{s} \leq \frac{1}{p} $
					it holds that
					\begin{equation}
					\begin{split}
					&
					\sup_{
						\substack{
							u_1, \ldots, u_m \in \L^s( \lambda_\O; \R^k ), 
							\\
							\prod^m_{i=1}
							\| u_i \|_{\L^s(\lambda_\O; \R^k ) }
							> 0
						}
					}
					\Bigg(
					\frac{
						\| ( f^{(m)} \circ v - f^{(m)} \circ w ) ( u_1, \ldots, u_m ) 
						\|_{\L^p( \lambda_\O; \R^l ) }
					}{
					\| u_1 \|_{ \L^s( \lambda_\O; \R^k ) }
					\cdot
					\ldots
					\cdot
					\| u_m \|_{ \L^s( \lambda_\O; \R^k ) }
				}
				\Bigg)
				\\
				&
				\leq
				\left[
				\sup_{ \substack{
						x, y \in \R^k,
						\\
						x \neq y
					}}
					\frac{ 
						\| f^{(m)}(x) - f^{(m)}(y) \|_{L^{(m)}(\R^k, \R^l )}
					}
					{
						\| x - y \|_{\R^k}
					}
					\right]
					\!
					[ \lambda_{\R^d} ( \O ) ]^{ [ \frac{1}{p} - \frac{1}{r} - \frac{m}{s} ] }
					\| v - w \|_{\L^r( \lambda_\O; \R^k ) }
					.
					\end{split}
					\end{equation}
					This and item~\eqref{item:2}
					establish that for all
					$ m \in \{ 1, 2, \ldots, n \} $,
					$ r, s \in (p,\infty) $,
					$ v, w \in L^{  \max \{ r, q \} } ( \lambda_\O ; \R^k ) $
					with
					$ \frac{1}{r} + \frac{m}{s} \leq \frac{1}{p} $
					it holds that
					\begin{equation}
					\begin{split}
					&
					\sup_{ 
						u_1, \ldots, u_m \in L^{ \max \{ s, q \} }( \lambda_\O; \R^k ) \backslash \{ 0 \}
					}
					\Bigg(
					\frac{
						\| F^{(m)}(v) ( u_1, \ldots, u_m ) 
						-
						F^{(m)}(w) ( u_1, \ldots, u_m ) 
						\|_{L^p(\lambda_\O; \R^l ) }
					}{
					\| u_1 \|_{L^s ( \lambda_\O ; \R^k ) }
					\cdot
					\ldots
					\cdot
					\| u_m \|_{L^s ( \lambda_\O ; \R^k ) }
				}
				\Bigg)
				\\
				&
				\leq
				\left[
				\sup_{ \substack{
						x, y \in \R^k,
						\\
						x \neq y
					}}
					\frac{ 
						\| f^{(m)}(x) - f^{(m)}(y) \|_{L^{(m)}(\R^k, \R^l )}
					}
					{
						\| x - y \|_{\R^k}
					}
					\right]
					\!
					[ \lambda_{\R^d} ( \O ) ]^{ [ \frac{1}{p} - \frac{1}{r} - \frac{m}{s} ] }
					\| v - w \|_{ L^r( \lambda_\O; \R^k ) }
					.
					\end{split}
					\end{equation}
					This proves item~\eqref{item:4}.
					Item~\eqref{item:5} 
					is an immediate consequence of item~\eqref{item:4}.
					The proof of Proposition~\ref{proposition:nonlinearityF}
					is thus completed.
				\end{proof}
\subsection{Regular diffusion coefficients}
\label{subsection:Regular_diffusion}
\begin{lemma}
	\label{lemma:InterpolationSpaces}
	Consider the notation in Subsection~\ref{sec:notation}, 
	let
	$ p \in [ 2,\infty) $, 
	$ r \in ( \nicefrac{1}{4}, \infty) $,
	let $ ( H, \langle \cdot, \cdot \rangle_H, \left \| \cdot \right \|_H ) 
	=
	( L^2(\lambda_{(0,1)}; \R ), 
	\langle \cdot, \cdot \rangle_{ L^2(\lambda_{(0,1)}; \R ) },
	\left \| \cdot \right \|_{ L^2(\lambda_{(0,1)}; \R ) } )
	$,
	and
	let $ A \colon D(A) \subseteq H \to H $
	be the Laplacian with Dirichlet boundary conditions on $ H $.
	Then
	\begin{enumerate}[(i)]
		\item \label{item:i} it holds for all $ v \in H $
		that $ ( - A)^{-r} v \in L^p ( \lambda_{(0,1)};  \R) $,
		\item \label{item:ii} it holds that
		$ ( H \ni v \mapsto (-A)^{-r} v \in L^p ( \lambda_{(0,1)};  \R) )
		\in \gamma ( H, L^p ( \lambda_{(0,1)};  \R) ) $,
		and
		\item \label{item:iii} it holds that
		\begin{equation}
		\begin{split}
		&
		\| H \ni v \mapsto (-A)^{- r} v \in L^p ( \lambda_{(0,1)};  \R) \|_{\gamma( H, L^p( \lambda_{(0,1)}; \R) ) }
		\\
		&
		\leq
		\bigg[
		\int_\R
		\tfrac{ | x |^p }{\sqrt{2\pi}} \, e^{ - \nicefrac{x^2}{2} } 
		\, dx
		\bigg]^{\nicefrac{1}{p}}
		\bigg[
		\smallsum\limits_{ n = 1 }^\infty
		\tfrac{ 1 }{ n ^{4r} } 
		\bigg]^{\nicefrac{1}{2}}
		< \infty.
		\end{split}
		\end{equation}
	\end{enumerate}
\end{lemma}
\begin{proof}[Proof of Lemma~\ref{lemma:InterpolationSpaces}]
	Throughout this proof let
	$ ( \Omega, \mathcal{F}, \P ) $
	be a probability space,
	let $ \gamma_n \colon \Omega \to \R $, $ n \in \N $,
	be independent standard normal random variables,
	let $ f_n \colon (0,1) \to \R $, $ n \in \N $,
	satisfy for all $ n \in \N $, $ x \in (0,1) $
	that
	$ f_n(x) = \sqrt{2} \sin(n \pi x ) $,
	let $ ( \rho_n )_{ n \in \N } \subseteq \R $
	satisfy for all
	$ n \in \N $ that $ \rho_n = \pi^2 n^2 $,
	and let $ e_n \in H $, $ n \in \N $,
	satisfy for all $ n \in \N $ that
	$ e_n = [ f_n ]_{ \lambda_{(0,1)}, \B(\R) } $.
	Note that item~\eqref{item:i} is an immediate 
	consequence from the Sobolev embedding theorem.
	It thus remains to prove item~\eqref{item:ii} and~\eqref{item:iii}.
	For this observe that Jensen's inequality ensures that for all
	$ M, N \in \N $ with $ M \leq N $ it holds that
	\begin{equation}
	\begin{split}
	&
	\E \Bigg[
	\bigg\|
	\smallsum\limits_{n=M}^N \gamma_n (-A)^{- r} e_n
	\bigg\|_{L^p( \lambda_{(0,1)}; \R ) }^2
	\Bigg]
	=
	\E \Bigg[
	\bigg\|
	\smallsum\limits_{n=M}^N \gamma_n (\rho_n)^{-r} e_n
	\bigg\|_{L^p( \lambda_{(0,1)}; \R ) }^2
	\Bigg]
	\\
	&
	\leq
	\Bigg(
	\E \Bigg[
	\bigg\|
	\smallsum\limits_{n=M}^N \gamma_n (\rho_n)^{-r} e_n
	\bigg\|_{L^p( \lambda_{(0,1)}; \R ) }^p
	\Bigg]
	\Bigg)^{\nicefrac{2}{p}}
	\end{split}
	\end{equation}
	This implies that  
	for all
	$ M, N \in \N $ with $ M \leq N $ it holds that
	\begin{equation}
	\begin{split}
	&
	\E \Bigg[
	\bigg\|
	\smallsum\limits_{n=M}^N \gamma_n (-A)^{-r} e_n
	\bigg\|_{L^p( \lambda_{(0,1)}; \R ) }^2
	\Bigg]
	\\
	&
	\leq
	\bigg[
	\int_0^1
	\E\bigg[
	\bigg|
	\smallsum\limits_{n=M}^N
	\gamma_n ( \rho_n )^{-r} f_n(x)
	\bigg|^p
	\bigg]
	\,
	dx
	\bigg]^{\nicefrac{2}{p}}
	\\
	&
	=
	\bigg[
	\int_0^1
	\bigg|
	\E\bigg[
	\bigg|
	\smallsum\limits_{n=M}^N
	\gamma_n ( \rho_n )^{-r} f_n(x)
	\bigg|^2
	\bigg]
	\bigg|^{ \nicefrac{p}{2} }
	\E \big[ |\gamma_1 |^p \big]
	\,
	dx
	\bigg]^{\nicefrac{2}{p}}
	\\
	&
	=
	\bigg[
	\int_0^1
	\bigg(
	\smallsum\limits_{n=M}^N
	( \rho_n )^{-2r}
	| f_n(x) |^2
	\bigg)^{\nicefrac{p}{2}}
	\E \big[ | \gamma_1 |^p \big]
	\, dx
	\bigg]^{\nicefrac{2}{p}}
	\\
	&
	=
	\bigg[
	\int_0^1
	\bigg(
	\smallsum\limits_{n=M}^N
	( \rho_n )^{-2r}
	| f_n(x) |^2
	\bigg)^{\nicefrac{p}{2}}
	\, dx
	\bigg]^{\nicefrac{2}{p}}
	\| \gamma_1 \|_{ \L^p( \P; \R )}^2
	.
	\end{split}
	\end{equation}
	The Minkowski inequality hence shows that for all
	$ M, N \in \N $ with $ M \leq N $ it holds that
	\begin{equation}
	\begin{split}
	&
	\E \Bigg[
	\bigg\|
	\smallsum\limits_{n=M}^N \gamma_n (-A)^{-r} e_n
	\bigg\|_{L^p( \lambda_{(0,1)}; \R ) }^2
	\Bigg]
	\leq
	\| \gamma_1 \|_{ \L^p( \P; \R )}^2
	\bigg\|
	\smallsum\limits_{n=M}^N
	( \rho_n )^{-2r} | f_n |^2
	\bigg\|_{\L^{\nicefrac{p}{2}}( \P; \R)}
	\\
	&
	\leq
	\| \gamma_1 \|_{ \L^p( \P; \R )}^2
	\bigg(
	\smallsum\limits_{n=M}^N
	( \rho_n )^{-2r}
	\| f_n \|_{\L^p(\P; \R) }^2
	\bigg)
	.
	\end{split}
	\end{equation}
	This proves that for all $ M, N \in \N $
	with $ M \leq N $
	it holds that
	\begin{equation}
	\begin{split}
	& 
	\bigg\|
	\smallsum\limits_{n=M}^N \gamma_n (-A)^{-r} e_n
	\bigg\|_{ L^2( \P; L^p( \lambda_{(0,1)}; \R ) ) }^2
	\leq
	\| \gamma_1 \|_{ \L^p( \P; \R )}
	\bigg[
	\smallsum\limits_{n=M}^N
	( \rho_n )^{-2r}
	\| f_n \|_{\L^p(\P; \R)}^2
	\bigg]^{\nicefrac{1}{2}}
	\\
	&
	\leq
	\| \gamma_1 \|_{ \L^p( \P; \R )}
	\bigg[
	2
	\smallsum\limits_{n=M}^N 
	( \rho_n )^{-2r}
	\bigg]^{\nicefrac{1}{2}}
	=
	\| \gamma_1 \|_{ \L^p( \P; \R )}
	\tfrac{ \sqrt{2} }{ \pi^{2r} }
	\bigg[
	\smallsum\limits_{n=M}^N n^{-4r}
	\bigg]^{ \nicefrac{1}{2}}
	\\
	&
	\leq
	\| \gamma_1 \|_{ \L^p( \P; \R )} 
	\bigg[
	\smallsum\limits_{n=M}^N n^{-4r}
	\bigg]^{ \nicefrac{1}{2}}
	=
	\bigg[
	\int_\R
	\tfrac{ | x |^p }{\sqrt{2\pi}} \, e^{ - \nicefrac{x^2}{2} } \, dx
	\bigg]^{\nicefrac{1}{p}}
	\bigg[ \sum\limits_{n=M}^N n^{-4r} \bigg]^{ \nicefrac{1}{2} } < \infty.
	\end{split}
	\end{equation}
	This and, e.g, \cite[Theorem~3.20]{VanNeerven2010Gamma} completes the proof of Lemma~\ref{lemma:InterpolationSpaces}.
\end{proof}
\begin{lemma}
	\label{lemma:B}
	Consider the notation in Subsection~\ref{sec:notation},
	let $ d \in \N $,
	$ p \in (2,\infty) $,
	$ \beta \in (- \infty, - \frac{d}{2p} ] $,
	let $ ( H, \langle \cdot, \cdot \rangle_H, \left \| \cdot \right \|_H ) 
	=
	( L^2(\lambda_{(0,1)^d}; \R), 
	\langle \cdot, \cdot \rangle_{ L^2(\lambda_{(0,1)^d}; \R ) },
	\left \| \cdot \right \|_{ L^2(\lambda_{(0,1)^d}; \R ) } )
	$,
	let $ A \colon D(A) \subseteq H \to H $
	be the Laplacian with Dirichlet boundary conditions on $ H $,
	and let
	$ ( H_r, \langle \cdot, \cdot \rangle_{H_r}, 
	\left \| \cdot \right \|_{H_r} ) $, $ r \in \R $,
	be a family of interpolation spaces
	associated to $ - A $.
	Then
	\begin{enumerate}[(i)]
		\item\label{item:exist.multiplicator} there exists a unique bounded linear operator
		$ B \in L( L^{ p } ( \lambda_{(0,1)^d}; \R) , L ( H, H_\beta ) ) $
		which satisfies for all
		$ v \in L^{ \max \{p, 4 \} }( \lambda_{(0,1)^d}; \R ) $,
		$ u \in L^4( \lambda_{(0,1)^d}; \R ) $
		that
		\begin{equation} 
		( B v ) u = v \cdot u 
		\end{equation}
		and
		\item it holds that
		\begin{equation}
		\| B \|_{ L ( L^p( \lambda_{(0,1)^d}; \R), L (H, H_\beta) ) }
		\leq
		\sup_{ w \in H_{-\beta} \backslash \{0\} } 
		\bigg[
		\tfrac{ \| w \|_{ L^{2p/(p-2)}( \lambda_{(0,1)^d}; \R ) } }
		{ \| w \|_{ H_{ -\beta} } }
		\bigg]
		< \infty.
		\end{equation}
	\end{enumerate}
\end{lemma}
\begin{proof}[Proof of Lemma~\ref{lemma:B}]
	Throughout this proof let
	\begin{equation} 
	M \colon L^{ \max \{p, 4 \} }( \lambda_{(0,1)^d}; \R) 
	\to
	L ( L^4( \lambda_{(0,1)^d}; \R ), H ) 
	\end{equation}
	be the function which satisfies for all
	$ v \in  L^{ \max \{p, 4 \} }( \lambda_{(0,1)^d}; \R)  $,
	$ u \in  L^4( \lambda_{(0,1)^d}; \R) $
	that
	$ M ( v ) u = v \cdot u $.
	Observe that for all 
	$ v \in  L^{ \max \{p, 4 \} }( \lambda_{(0,1)^d}; \R)  $,
	$ u \in  L^4( \lambda_{(0,1)^d}; \R) $
	it holds that
	\begin{equation}
	\begin{split}
	\| ( M ( v ) ) u \|_{H_\beta}
	&
	=
	\| v \cdot u \|_{H_\beta}
	= 
	\| ( - A )^\beta ( u \cdot v ) \|_H
	=
	\sup_{ w \in H \backslash \{0\} }
	\bigg[
	\tfrac{
		| \langle w, ( - A )^\beta ( v \cdot u ) \rangle_H |
	}{
	\| w \|_H
}
\bigg]
\\
&
=
\sup_{ w \in H \backslash \{ 0 \} }
\bigg[
\frac{
	| \langle ( - A )^\beta w, v \cdot u \rangle_H |
}{
\| ( - A )^{ - \beta } ( - A )^\beta w \|_H
}
\bigg]
.
\end{split}
\end{equation}
H\"older's inequality hence ensures that for all
$ v \in  L^{ \max \{p, 4 \} }( \lambda_{(0,1)^d}; \R)  $,
$ u \in  L^4( \lambda_{(0,1)^d}; \R) $
it holds that
\begin{equation}
\begin{split}
\| ( M ( v ) ) u \|_{H_\beta}
&
=
\sup_{ w \in H_{ - \beta } \backslash \{ 0 \} }
\bigg[
\tfrac{
	| \langle w, v \cdot u \rangle_H |
}{
\| ( - A)^{-\beta} w \|_H
}
\bigg]
\\
&
\leq
\sup_{ w \in H_{-\beta} \backslash \{ 0 \} }
\Bigg[
\tfrac{
	\| w \|_{ L^{1/ ( 1/2-1/p) } ( \lambda_{ (0,1)^d}; \R ) }
	\| v \|_{ L^p ( \lambda_{ (0,1)^d}; \R ) }
	\| u \|_H
}{
\| ( - A )^{-\beta} w \|_H
}
\Bigg]
\\
&
=
\bigg[
\sup_{ w \in H_{ - \beta } \backslash \{ 0 \} }
\tfrac{
	\| w \|_{ L^{ 2p/(p-2) } ( \lambda_{ (0,1)^d }; \R ) }
}{
\| ( - A )^{-\beta} w \|_H
}
\bigg]
\| v \|_{ L^p ( \lambda_{ (0,1)^d }; \R ) }
\| u \|_H
.
\end{split}
\end{equation}
Combining this and the Sobolev embedding theorem with the fact that
\begin{equation} 
( - 2 \beta ) - 0 = - 2 \beta \geq \tfrac{d}{p}
=
d \big [ \tfrac{ 1 }{ 2 } - [ \tfrac{1}{2} - \tfrac{1}{p} ] \big]
=
d \big [ \tfrac{ 1 }{ 2 } - \tfrac{1}{(2p/(p-2))} \big] 
\end{equation}
proves that
for all
$ v \in  L^{ \max \{p, 4 \} }( \lambda_{(0,1)^d}; \R)  $
it holds that
\begin{equation}
\begin{split}
\sup_{ u \in L^4 ( \lambda_{(0,1)^d}; \R) \backslash \{0 \} }
\bigg[
\tfrac{ \| ( M ( v ) ) u \|_{H_\beta} }
{ \| u \|_H  }
\bigg]
\leq
\underbrace{
	\bigg[
	\sup_{ w \in H_{ - \beta } \backslash \{ 0 \} }
	\tfrac{
		\| w \|_{ L^{ 2p/(p-2) } ( \lambda_{ (0,1)^d }; \R ) }
	}{
	\| ( - A )^{-\beta} w \|_H
}
\bigg]}_{< \infty}
\| v \|_{ L^p ( \lambda_{ (0,1)^d }; \R ) }
.
\end{split}
\end{equation}
This implies that there exists a unique function
$ \mathcal{M} \colon L^{ \max \{p, 4 \} } ( \lambda_{(0,1)^d}; \R) \to L ( H, H_\beta ) $
which satisfies for all
$ v \in L^{ \max \{p, 4\} }  ( \lambda_{(0,1)^d}; \R) $,
$ u \in L^4 ( \lambda_{(0,1)^d}; \R) $ that
\begin{equation}
\label{eq:defM1}
( \mathcal{M}(v))u
=
( M ( v ) )( u ) 
= v \cdot u
\end{equation}
and
\begin{equation}
\| \mathcal{M}( v ) \|_{L ( H, H_\beta)}
\leq
\bigg[
\sup_{ w \in H_{-\beta} \backslash \{0 \} }
\tfrac{ \| w \|_{L^{2p/(p-2)}( \lambda_{(0,1)^d}; \R ) } }
{ \| (-A)^{-\beta} w \|_H }
\bigg]
\| v \|_{ L^p ( \lambda_{ (0,1)^d}; \R ) } 
< \infty .
\end{equation}
This, in turn, assures that there exists a unique bounded linear
operator
\begin{equation} 
B \in L ( L^p ( \lambda_{ (0,1) }; \R ), L ( H, H_\beta)  ) 
\end{equation}
which satisfies for all
$ v \in L^{ \max \{p, 4\} }  ( \lambda_{(0,1)^d}; \R) $
that
\begin{equation}
\label{eq:defM2}
B(v) = \mathcal{M}(v)
\end{equation}
and
\begin{equation}
\label{eq:defM3}
\| B \|_{L( L^p ( \lambda_{ (0,1) }; \R ), L ( H, H_\beta)  ) }
\leq 
\sup_{ w \in H_{-\beta} \backslash \{0 \} }
\Bigg[
\tfrac{ \| w \|_{L^{2p/(p-2)}( \lambda_{(0,1)^d}; \R ) } }
{ \| (-A)^{-\beta} w \|_H }
\Bigg]
< \infty . 
\end{equation}
Combining~\eqref{eq:defM1}, \eqref{eq:defM2}, and~\eqref{eq:defM3}
completes the proof of Lemma~\ref{lemma:B}.
\end{proof}
\begin{lemma}
\label{lem:eps.beta.radonifying}
Let $ \lambda_{(0,1)} \colon
\B ( ( 0,1) ) \to [0,\infty] $
be the Lebesgue-Borel
measure on $ (0,1) $,
	let $p\in[2,\infty)$, 
	$\varepsilon\in[0,\infty)$, 
	$\beta\in(-\infty,-\nicefrac{1}{4}- \varepsilon)$, 
	$ 
	( H, \langle \cdot, \cdot \rangle_H, \left \| \cdot \right \|_H ) 
	=
	( L^2(\lambda_{(0,1)}; \R), 
	\langle \cdot, \cdot \rangle_{ L^2(\lambda_{(0,1)}; \R ) } $,
	$
	\left \| \cdot \right \|_{ L^2(\lambda_{(0,1)}; \R ) } )
	$, 
	$
	(V,\left\|\cdot\right\|_V)
	=
	(\lpnb{p}{\lambda_{(0,1)}}{\R},
	\left\|\cdot\right\|_{\lpnb{p}{\lambda_{(0,1)}}{\R}})
	$,
	let $A\colon D(A)\subseteq H \to H$ be the Laplacian with Dirichlet boundary conditions on $H$, 
	let 
	$ ( H_r, \langle \cdot, \cdot \rangle_{H_r}, 
	\left \| \cdot \right \|_{H_r} ) $, $ r \in \R $,
	be a family of interpolation spaces
	associated to $ - A $, 
		let $ \mathcal{A}\colon D(\mathcal{A})\subseteq V \to V $ be the Laplacian with Dirichlet boundary conditions on $V$, 
	and let 
	$
	(V_r,\left\|\cdot\right\|_{V_r})
	$, 	$ r \in \R $, 
	be a family of interpolation spaces associated to $ -\mathcal{A} $.
	Then
	\begin{enumerate}[(i)]
	\item
	\label{item:H.V.embedding}
	there exists a unique continuous function 
	$\iota\colon H_{-\varepsilon} \to V_\beta$ 
	which satisfies for all 
	$ v \in V $ 
	that
	$ \iota (v) = v $,
	\item
	\label{item:H.V.radonifying}
	it holds that $ \iota \in \gamma(H_{-\varepsilon},V_\beta) $,
	and
	\item \label{item:gamma_norm}
	it holds that
	\begin{equation}
	\| \iota \|_{ \gamma(H_{-\varepsilon}, V_\beta ) } 
		\leq
		\bigg[
		\int_\R
		\tfrac{ | x |^p }{\sqrt{2\pi}} \, e^{ - \nicefrac{x^2}{2} } 
		\, dx
		\bigg]^{\nicefrac{1}{p}}
		\bigg[
		\smallsum\limits_{ n = 1 }^\infty
		n^{4(\beta+\varepsilon)}  
		\bigg]^{\nicefrac{1}{2}}
		< \infty
		.
	\end{equation}
	\end{enumerate}
\end{lemma}
\begin{proof}[Proof of Lemma~\ref{lem:eps.beta.radonifying}]
	Throughout this proof
	let $ \varphi \in L(H_{-\varepsilon}, H) $
	be the unique
	bounded linear operator which satisfies for all
	$ v \in H $ that
	\begin{equation}
	\varphi( v ) = (-A)^{-\varepsilon} v
	\end{equation}
	and let
	$ \phi \in L(V,V_\beta) $
	be the unique bounded linear operator
	which satisfies
	for all
	$ v \in V_{-\beta} $
	that
	\begin{equation}
	\phi(v) = (- \mathcal{A} )^{-\beta } v .
	\end{equation}
	Observe that Lemma~\ref{lemma:InterpolationSpaces}
	and the assumption that
	$ \beta + \varepsilon < - \nicefrac{1}{4} $
	prove
	\begin{enumerate}[(a)]
		\item \label{item:aaa} that $ \forall \, v \in H \colon 
		(-A)^{\beta + \varepsilon} v \in V $,
		\item \label{item:bbb} that
		$ ( H \ni v \mapsto (-A)^{\beta+ \varepsilon} v \in V ) 
		\in \gamma(H, V) $,
		and
		\item \label{item:ccc} that
		\begin{equation}
		\!\!
		\| 
		H \ni v \mapsto (-A)^{\beta + \varepsilon} v \in V
		\|_{\gamma(H, U)}
		\leq
		\bigg[
		\int_\R
		\tfrac{ |x|^p }{ \sqrt{2 \pi} }
		e^{-\nicefrac{x^2}{2}}
		\, dx 
		\bigg]^{\nicefrac{1}{2}}
		\bigg[
		\smallsum\limits_{n=1}^\infty
		n^{4(\beta + \varepsilon)}
		\bigg]^{\nicefrac{1}{2}}
		< \infty 
		.
		\end{equation}
	\end{enumerate}
	Note that item~\eqref{item:aaa} assures that there exist
	functions
	$ \Phi \colon H \to V $
	and 
	$ \iota \colon H_{-\varepsilon} \to V_\beta $
	which satisfy for all $ v \in H $ that
	\begin{equation}
	\Phi(v) = (-A)^{ \beta + \varepsilon } v
	\end{equation}
	and 
	\begin{equation}
	\iota = \phi \circ \Phi \circ \varphi .
	\end{equation}
	Observe that item~\eqref{item:bbb}
	and item~\eqref{item:ccc} 
	establish that
	$ \Phi \in \gamma(H, V) $
	and
	\begin{equation}
	\| \Phi \|_{\gamma(H,V)}
	\leq
	\bigg[
	\int_\R
	\tfrac{ |x|^p }{ \sqrt{2 \pi} }
	e^{-\nicefrac{x^2}{2}}
	\, dx 
	\bigg]^{\nicefrac{1}{2}}
	\bigg[
	\smallsum\limits_{n=1}^\infty
	n^{4(\beta + \varepsilon)}
	\bigg]^{\nicefrac{1}{2}}
	< \infty .
	\end{equation}
	Combining this, the fact that
	$ \varphi \in L(H_{-\varepsilon}, H) $,
	and the fact that
	$ \phi \in L(V, V_\beta) $
	with Lemma~\ref{lemma:multiplication_estimate} 
	ensures that
	$ \iota \in \gamma( H_{-\varepsilon}, V_\beta ) $
	and
	\begin{equation}
	\begin{split}
	\label{eq:Iota_estimate}
	\| \iota \|_{ \gamma(H_{-\varepsilon}, V_\beta)}
	&
	\leq
	\| \phi \|_{L(V, V_\beta)}
	\| \Phi \|_{\gamma(H,V)}
	\| \varphi \|_{L(H_{-\varepsilon}, H)}
	\\
	&
	=
	\| \Phi \|_{\gamma(H, V)}
	\leq
	\bigg[
	\int_\R
	\tfrac{ |x|^p }{ \sqrt{2 \pi} }
	e^{-\nicefrac{x^2}{2}}
	\, dx 
	\bigg]^{\nicefrac{1}{2}}
	\bigg[
	\sum_{n=1}^\infty
	n^{4(\beta + \varepsilon)}
	\bigg]^{\nicefrac{1}{2}}
	< \infty .
	\end{split}
	\end{equation}
	Next note that the fact that $ \forall \, v \in V,
	t \in [0,\infty) \colon 
	e^{tA} v = e^{ t \mathcal{A} } v $,
	e.g., \cite[item~(ii) of Theorem~1.10 in Chapter~II]{en00} and, e.g., \cite[Definition~5.25 in Chapter~II]{en00} 
	ensure that for all
	$ v \in V $ it holds that
	\begin{equation}
	( - A)^\beta v = ( - \mathcal{A} )^\beta v 
	.
	\end{equation}
	Hence, we obtain for all $ v \in V $ that
	\begin{equation}
	\begin{split}
	\iota(v) 
	&
	= 
	\phi( \Phi( \varphi( v ) ) ) 
	=
	\phi( (-A)^{\beta + \varepsilon} ( - A)^{-\varepsilon} v )
	=
	\phi( (-A)^\beta v ) 
	\\
	&
	=
	\phi( (-\mathcal{A} )^\beta v ) 
	=
	( - \mathcal{A} )^{-\beta} (- \mathcal{A} )^\beta v 
	=
	v. 
	\end{split}
	\end{equation}
	This and~\eqref{eq:Iota_estimate}
	complete the proof of Lemma~\ref{lem:eps.beta.radonifying}.
\end{proof}
\begin{proposition}
	\label{prop.nonlinearityB}
	Consider the notation in Subsection~\ref{sec:notation}, 
	let $ n \in \N $, $ \beta\in(-\infty,-\nicefrac{1}{4}) $, $ p\in( \max \{ \frac{n}{2(|\beta|-\nicefrac{1}{4})}, 2 n \} ,\infty) $, 
	$ 
	( H, \langle \cdot, \cdot \rangle_H, \left \| \cdot \right \|_H ) 
	=
	( L^2(\lambda_{(0,1)}; \R), 
	\langle \cdot, \cdot \rangle_{ L^2(\lambda_{(0,1)}; \R ) } $,
	$
	\left \| \cdot \right \|_{ L^2(\lambda_{(0,1)}; \R ) } )
	$, 
	$
	(V,\left\|\cdot\right\|_V)
	=
	(\lpnb{p}{\lambda_{(0,1)}}{\R},
	$
	$
	\left\|\cdot\right\|_{\lpnb{p}{\lambda_{(0,1)}}{\R}})
	$, 
	let $ b \colon \R \to \R $ be
	an $ n $-times continuously differentiable 
	function
	with globally bounded derivatives, 
	let $A\colon D(A)\subseteq H \to H$ be the Laplacian with Dirichlet boundary conditions on $H$, 
	let 
	$ ( H_r, \langle \cdot, \cdot \rangle_{H_r}, 
	\left \| \cdot \right \|_{H_r} ) $, $ r \in \R $,
	be a family of interpolation spaces
	associated to $ - A $, 
	let $ \mathcal{A}\colon D(\mathcal{A})\subseteq V \to V $ be the Laplacian with Dirichlet boundary conditions on $V$, 
	and let 
	$
	(V_r,\left\|\cdot\right\|_{V_r})
	$, 	$ r \in \R $, 
	be a family of interpolation spaces associated to $ -\mathcal{A} $.
	Then
	\begin{enumerate}[(i)]
		\item
		\label{item:exist.B}
		there exists a unique continuous function 
		$ B\colon V \to \gamma(\lpnb{2}{\lambda_{(0,1)}}{\R},V_\beta) $
		which satisfies for all $ u, v \in \L^{ 2p }(
		\lambda_{(0,1)}; \R ) $ 
		that 
		\begin{equation}
		B\big([v]_{\lambda_{(0,1)},\mathcal{B}(\R)}\big)
		[u]_{\lambda_{(0,1)},\mathcal{B}(\R)}
		=
		\big[\{b(v(x)) \cdot u(x)\}_{x\in(0,1)}\big]_{\lambda_{(0,1)},\mathcal{B}(\R)}
		,
		\end{equation}
		\item
		\label{item:B.smooth}
		it holds that $B$ is $n$-times continuously Fr\'{e}chet differentiable with globally bounded derivatives,
		\item \label{item:finite_constant}
		it holds for all
		$ \delta \in ( 
		\frac{1}{p}
		\max \{ 
		\frac{n}{2( |\beta |- 1/4)}
		,
		2n
		\}, 1) $
		that
		\begin{equation}
		\sup_{ w \in H_{
				\nicefrac{ n }{ (2 p \delta) }
			} \backslash \{ 0 \} }
		\Bigg[
		\tfrac{
			\| w \|_{
				L^{
					\nicefrac{2 p \delta }{ ( p \delta - 2 n ) }
				}( \lambda_{(0,1)}; \R)
			}
		}{
		\| w \|_{ H_{
				\nicefrac{ n }{ (2 p \delta) }
			} }
		}
		\Bigg]
		+
		\smallsum\limits_{ l = 1 }^\infty
		l^{4(\beta+\nicefrac{n}{(2p\delta)})}  
		< \infty 
		,
		\end{equation}
		\item
		\label{item:B.Cb}
		it holds for all $k\in\{1,\ldots,n\}$,
		$ \delta \in ( 
		\frac{1}{p}
		\max \{ 
		\frac{n}{2( |\beta |- 1/4)}
		,
		2n
		\}, 1) $ that 
		\begin{multline}
		\sup_{ v \in V}
		\|B^{(k)}(v)\|_{L^{(k)}(V,\gamma(H,V_\beta))} 
		\leq
		\bigg[
		\int_\R
		\tfrac{ | x |^p }{\sqrt{2\pi}} \, e^{ - \nicefrac{x^2}{2} } 
		\, dx
		\bigg]^{\nicefrac{1}{p}}
		\bigg[
		\smallsum\limits_{ l = 1 }^\infty
		l^{4(\beta+\nicefrac{n}{(2p\delta)})}  
		\bigg]^{\nicefrac{1}{2}}
		\\ 
		\cdot
		\Bigg[
		\sup_{ w \in H_{
				\nicefrac{ n }{ (2 p \delta) }
			} \backslash \{ 0 \} }
		\tfrac{
			\| w \|_{
				L^{
					\nicefrac{2 p \delta }{ ( p \delta - 2 n ) }
				}( \lambda_{(0,1)}; \R)
			}
		}{
		\| w \|_{ H_{
				\nicefrac{ n }{ (2 p \delta) }
			} }
		}
		\Bigg]
		\bigg[
		\sup_{ x \in \R }
		| b^{(k)} ( x ) | 
		\bigg]
		< \infty
		,
		\end{multline}
		and 
		\item
		\label{item:B:Lip}
		it holds for all $k\in\{1,\ldots,n\}$,
		$ \delta \in ( 
		\frac{1}{p}
		\max \{ 
		\frac{n}{2( |\beta |- 1/4)}
		,
		2n
		\}, 1) $,
		$ r \in [ \frac{p\delta}{n-k\delta}, \infty) $
		that 
		\begin{equation}
		\begin{split} 
		&
		\sup_{\substack{v,w \in \lpnb{\max\{r,p\}}{\lambda_{(0,1)}}{\R},\\ v\neq w}}
		\Bigg[
		\frac{
			\| B^{(k)}(v)
			-
			B^{(k)}(w)
			\|_{L^{(k)}(V,\gamma( H, V_\beta) 
				)}
		}{
		\|v-w\|_{\lpnb{r}{\lambda_{(0,1)}}{\R}}
	}
	\Bigg]
	\\&
	\leq
	\bigg[
	\int_\R
	\tfrac{ | x |^p }{\sqrt{2\pi}} \, e^{ - \nicefrac{x^2}{2} } 
	\, dx
	\bigg]^{\nicefrac{1}{p}}
	\bigg[
	\smallsum\limits_{ l = 1 }^\infty
	l^{4(\beta+\nicefrac{n}{(2p\delta)})}  
	\bigg]^{\nicefrac{1}{2}}
	\\
	&
	\quad
	\cdot
	\bigg[
	\sup_{ w \in H_{
			\nicefrac{ n }{ (2 p \delta) }
		} \backslash \{ 0 \} }
	\tfrac{
		\| w \|_{
			L^{
				\nicefrac{2 p \delta }{ ( p \delta - 2 n ) }
			}( \lambda_{(0,1)}; \R)
		}
	}{
	\| w \|_{ H_{
			\nicefrac{ n }{ (2 p \delta) }
		} }
	}
	\bigg] 
	\Bigg[
	\sup_{
		\substack{
			x, y \in \R, \\ x \neq y
		}
	} 
	\frac{
		| b^{(k)}(x) - b^{(k)}(y) |
	}{
	| x - y |
}
\Bigg]
.
\end{split}
\end{equation}		
\end{enumerate}
\end{proposition}
\begin{proof}[Proof of Proposition~\ref{prop.nonlinearityB}]
	Throughout this proof
	let
	$ \delta \in ( 
	\frac{1}{p}
	\max \{ 
	\frac{n}{2( |\beta |- 1/4)}
	,
	2n
	\}, 1) $
	and
	let $ \psi \colon V \to L^{\nicefrac{p\delta}{n}}(\lambda_{(0,1)}; \R) $
	be the function which satisfies for all
	$ v \in \L^p( \lambda_{(0,1)}; \R) $ that
		\begin{equation}
		\label{eq:Nr1}
		\psi([v]_{\lambda_{(0,1)},\mathcal{B}(\R)})
		=
		[
		\{ b(v(x)) \}_{x \in (0,1) }
		]_{ \lambda_{(0,1)}, \B(\R) }
		=
		[b \circ v]_{\lambda_{(0,1)},\mathcal{B}(\R)}
		.
		\end{equation}
Note that 	
item~\eqref{item:1} of Proposition~\ref{proposition:nonlinearityF}
(with
$ k = 1 $, $ l = 1 $, $ d = 1 $, $ n = n $,
$ p = \frac{ p \delta }{n} $,
$ q = p $,
$ \O = (0,1) $,
$ f = b $,
$ F = \psi $
in the notation of
item~\eqref{item:1} of
Proposition~\ref{proposition:nonlinearityF}) establishes that
\begin{equation}
 \label{eq:Nr2}
 \psi
 \in 
 C^n_b(V,\lpnb{\nicefrac{p \delta}{n}}{\lambda_{(0,1)}}{\R})
 .
\end{equation}
Moreover, observe that item~\eqref{item:3} 
of Proposition~\ref{proposition:nonlinearityF}
(with
$ k = 1 $,
$ l = 1 $,
$ d = 1 $,
$ n = n $,
$ p = \frac{p \delta}{n} $,
$ q = p $,
$ \O = (0,1) $,
$ f = b $,
$ F = \psi $,
$ m = k $,
$ r = p $
for
$ k \in \{ 1, \ldots, n \} $
in the notation of
item~\eqref{item:3} 
of
Proposition~\ref{proposition:nonlinearityF})
proves
that
for all
$ k \in \{ 1, \ldots, n \} $
it holds that
\begin{equation}
\begin{split}
\label{eq:77}
& 
\sup_{ v \in V }
\| \psi^{(k)}(v) \|_{L^{(k)}(V, L^{\nicefrac{p \delta}{n}} ( \lambda_{(0,1)}; \R  )  )} 
\leq
\sup_{ x \in \R }
| b^{(k)} ( x ) |
< \infty
.
\end{split}
\end{equation}		
In addition, we apply item~\eqref{item:4}
of
Proposition~\ref{proposition:nonlinearityF}
(with
$ k = 1 $,
$ l = 1 $,
$ d = 1 $,
$ n = n $,
$ p = \frac{p\delta}{n} $,
$ q = p $,
$ \O = (0,1) $,
$ f = b $,
$ F = \psi $,
$ m = k $,
$ r = r $,
$ s = p $,
$ v = v $,
$ w = w $
for
$ v, w \in \lpnb{\max\{r,p\}}{\lambda_{(0,1)}}{\R} $,
$ r \in [ \frac{p\delta}{n-k\delta}, \infty) $, 
$ k \in \{1, \ldots, n \} $
in the notation of
 item~\eqref{item:4}
 of
Proposition~\ref{proposition:nonlinearityF})
to obtain that for all
$ k \in \{ 1, \ldots, n\} $,
$ r \in [ \frac{p\delta}{n - k \delta }, \infty ) $ 
it holds that
\begin{equation}
\begin{split}
\label{eq:Nr3b}
& 
\sup_{\substack{v,w \in \lpnb{\max\{r,p\}}{\lambda_{(0,1)}}{\R},\\ v\neq w}}
\Bigg[
\frac{
	\| \psi^{(k)}(v)
	-
	\psi^{(k)}(w)
	\|_{L^{(k)}(V, L^{ \nicefrac{p\delta}{n}}(\lambda_{(0,1)}; \R)  )
	}
}{
\|v-w\|_{\lpnb{r}{\lambda_{(0,1)}}{\R}}
}
\Bigg]
\\
&
=
\sup_{\substack{v,w \in \lpnb{\max\{r,p\}}{\lambda_{(0,1)}}{\R},\\ v\neq w}} 
\sup_{  v_1, \ldots, v_k \in V \backslash \{0\}}
\Bigg[
\frac{	
	\| [ \psi^{(k)}(v) - \psi^{(k)}(w) ](v_1, \ldots, v_k)\|_{\lpnb{\nicefrac{p\delta}{n}}{\lambda_{(0,1)}}{\R}}
}
{
\|v-w\|_{\lpnb{r}{\lambda_{(0,1)}}{\R}} \cdot	\| v_1 \|_V \cdot \ldots \cdot \| v_k \|_V	
}
\Bigg]
\\
&
\leq
\sup_{
	\substack{
		x, y \in \R, \\ x \neq y
	}
} 
\Bigg[
\frac{
	| b^{(k)}(x) - b^{(k)}(y) |
}{
| x - y |
}
\Bigg]
.
\end{split}
\end{equation}
Moreover, note that for all
$ q \in [p,\infty) $,
$ v \in \L^q( \lambda_{(0,1)}; \R) $ it holds that
	\begin{equation}
	\begin{split}
	&
	\int_0^1
	| b ( v ( x ) ) |^q
	\, dx
	=
	\int_0^1
	\bigg|
	b(0)
	+
	\int_0^1
	b'(rv(x)) v(x) 
	\, dr
	\bigg|^q
	\, dx
	\\
	&
	\leq
	\int_0^1
	\bigg(
	| b(0) |
	+
	|v(x)|
	\sup_{ y \in \R } | b'( y ) |
	\bigg)^q
	\, dx
	\\
	&
	\leq
	2^{q-1}
	\int_0^1
	\bigg(
	| b(0) |^q
	+
	|v(x)|^q
	\sup_{ y \in \R } | b'( y ) |^q
	\bigg)
	\, dx
	\\
	&
	=
	2^{q-1} 
	\bigg(
	| b(0) |^q
	+
	\sup_{ y \in \R } | b'( y ) |^q
	\int_0^1
	|v(x)|^q
	\, dx
	\bigg)
	< \infty
	.
	\end{split}
	\end{equation}
	This proves that for all
	$ q \in [p,\infty) $, $ v \in L^q( \lambda_{(0,1)}; \R ) $
	it holds that
	\begin{equation}
	\label{eq:Nr3d}
	\psi(v) \in L^q(\lambda_{(0,1)}; \R ) 
	.
	\end{equation}
In the next step we observe that 
	Lemma~\ref{lemma:B} 
	(with
	$ d = 1 $,
	$ p = \frac{p \delta}{n} $,
	$ \beta = - \frac{ n }{ 2p\delta } $,
	$ A = A $
	in the notation of
	Lemma~\ref{lemma:B})
	assures that there exists a unique
	\begin{equation}
	\label{eq:Nr3X}
	M \in L ( L^{\nicefrac{p\delta}{n}}( \lambda_{(0,1)}; \R),
	L (H , H_{ - \nicefrac{n}{(2p\delta)} } ) ) 
	\end{equation}
	which satisfies for all
	$ v \in L^{ \max \{ \nicefrac{p\delta}{n}, 4 \} }
	( \lambda_{(0,1)}; \R) $,
	$ u \in L^4( \lambda_{(0,1)}; \R ) $
	 that
	\begin{equation}
	\label{eq:Nr3e}
	( M v ) u = v \cdot u 
	\end{equation}
	and
	\begin{equation}
	\label{eq:Nr3f}
	\| M \|_{ L ( L^{ \nicefrac{p\delta}{n}} ( \lambda_{(0,1)} ; \R),
		L( H, H_{ - \nicefrac{n}{(2p\delta)} } ) ) }
	\leq
		\sup_{ w \in H_{
				\nicefrac{ n }{ (2 p \delta) }
			} \backslash \{ 0 \} }
		\bigg[
		\tfrac{
			\| w \|_{
				L^{
					\nicefrac{2 p \delta }{ ( p \delta - 2 n ) }
				}( \lambda_{(0,1)}; \R)
			}
		}{
		\| w \|_{ H_{
				\nicefrac{ n }{ (2 p \delta) }
			} }
		}
		\bigg]
		<
		\infty
		.
	\end{equation}
Moreover, we note that H\"older's inequality shows that for all
$ u, v \in L^{2p}( \lambda_{(0,1)}; \R ) $
it holds that
\begin{equation}
\label{eq:Nr4}
( M v ) u \in V.
\end{equation}
Furthermore, we observe that 
 Lemma~\ref{lem:eps.beta.radonifying} 
	(with
	$ p = p $,
	$ \varepsilon = \frac{n}{2p\delta} $,
	$ \beta = \beta $,
	$ A = A $,
	$ \mathcal{A} = \mathcal{A} $
	in the notation of
	Lemma~\ref{lem:eps.beta.radonifying})
	and the fact that 
	\begin{equation}
	\beta + \tfrac{n}{(2p\delta)} 
	=
	- | \beta | + \tfrac{n}{(2p\delta)} < -\tfrac{1}{4}
	\end{equation} 
	yield that there exists a unique 
	\begin{equation}
	\label{eq:Unique_I}
	 \iota \in \gamma ( H_{- \nicefrac{n}{(2p\delta)}}, V_\beta ) 
	 \end{equation} 
		which satisfies for all 
		$ v \in V $ that 
		\begin{equation}
		\label{eq:Nr6}
		\iota ( v ) = v
		\end{equation}
		and
		\begin{equation}
		\label{eq:I_estimate}
		\| \iota \|_{ \gamma(H_{-\nicefrac{n}{(2p\delta)}},V_\beta) }
		\leq  
		\bigg[
		\int_\R
		\tfrac{ | x |^p }{\sqrt{2\pi}} \, e^{ - \nicefrac{x^2}{2} } 
		\, dx
		\bigg]^{\nicefrac{1}{p}}
		\bigg[
		\smallsum\limits_{ l = 1 }^\infty
		l^{4(\beta+\nicefrac{n}{(2p\delta)})}  
		\bigg]^{\nicefrac{1}{2}}
		< \infty
		.
		\end{equation}
		In addition, note that~\eqref{eq:Nr1},
		\eqref{eq:Nr3d},
		\eqref{eq:Nr3e},
		\eqref{eq:Nr4},
		and~\eqref{eq:Nr6}
		demonstrate that for all 
		$ u, v \in \L^{2p}( \lambda_{(0,1)}; \R) $
		it holds that 
		\begin{equation} 
		\label{eq:defined_on_dense_subset}
		\iota \big( M(\psi([v]_{\lambda_{(0,1)}, \B(\R) }))[ u ]_{ \lambda_{(0,1)}, \B( \R)}
		\big)
		=
		\big[
		\{
		b(v(x)) \cdot u (x) 
		\}_{ x \in (0,1) }
		\big]_{ \lambda_{(0,1)}, \B(\R) }
		.
		\end{equation}
		Next observe that
		Lemma~\ref{lemma:multiplication_estimate},
		\eqref{eq:Nr3X},
		\eqref{eq:Nr3f},
		\eqref{eq:Unique_I},
		and~\eqref{eq:I_estimate}
		establish that
		\begin{enumerate}[(a)]
		\item  
		for all
		$ v \in L^{ \nicefrac{p\delta}{n}}( \lambda_{(0,1)}; \R ) $
		it holds that
		$
		\iota \circ [ M( v ) ]
		=
		\iota M ( v ) \in \gamma (H, V_\beta)
		$
		and
		\begin{equation}  
		\begin{split} 
		\label{eq:Nr9}
		\| \iota M(v) \|_{\gamma(H, V_\beta)}
		\leq
		\| \iota \|_{\gamma( H_{- \nicefrac{n}{(2p\delta)}}, V_\beta )}
		\| M(v) \|_{ L ( H, H_{- \nicefrac{n}{(2p\delta)}} )}
		< \infty 
		\end{split}
		\end{equation}
		and
		\item 
		that
			\begin{equation}
			\label{eq:110}
			(
			L^{\nicefrac{p\delta}{n}}(\lambda_{(0,1)}; \R) \ni w
			\mapsto
			\iota M(w) \in \gamma( H, V_\beta ) 
			)
			\in
			L(
			L^{\nicefrac{p\delta}{n}}(\lambda_{(0,1)}; \R),
			\gamma(H, V_\beta)
			).
			\end{equation}
	\end{enumerate}
Combining this with~\eqref{eq:Nr2}
and the chain rule for differentiation 
implies that there exists a unique function
\begin{equation}
\label{eq:Nr100}
B \in C_b^n( V, \gamma(H, V_\beta) ) 
\end{equation}
which satisfies for all $ v \in V $, $ u \in H $ that
\begin{equation}
\label{eq:Nr101}
B(v) u = \iota \big( M( \psi( v ) ) u \big)
.
\end{equation}
This and~\eqref{eq:defined_on_dense_subset}
prove items~\eqref{item:exist.B}
and~\eqref{item:B.smooth}.
Next observe that~\eqref{eq:Nr3f}
and~\eqref{eq:I_estimate}
establish item~\eqref{item:finite_constant}.
It thus remains to prove items~\eqref{item:B.Cb}
and~\eqref{item:B:Lip}.
For this note that~\eqref{eq:Nr2}, \eqref{eq:110},
and the chain rule for differentiation
assure that
for all
$ k \in \{ 1, \ldots, n \} $,
$ v, v_1, \ldots, v_k \in V $,
$ u \in H $
it holds that
\begin{equation}
\label{eq:112} 
	B^{(k)}(v)(v_1,\ldots,v_k) (u)
	=
	\iota M(\psi^{(k)}(v)(v_1,\ldots,v_k))u
	. 
\end{equation}
Therefore, we obtain
that for all 
$ k \in \{1,\ldots,n\} $,
$ v \in V $ 
it holds that 
\begin{equation}
\begin{split}
& 
\|B^{(k)}(v)\|_{L^{(k)}(V,\gamma(H,V_\beta))}
=
\sup_{ v_1, \ldots, v_k \in V \backslash \{0 \} }
\frac{
	\| B^{(k)}(v) (v_1, \ldots, v_k )
	\|_{\gamma( H, V_\beta) 
	}
}{
\| v_1 \|_V \cdot \ldots \cdot \| v_k \|_V
}
\\
&
=
\sup_{ v_1, \ldots, v_k \in V \backslash \{0\} }
\frac{
	\| 
	\iota M(\psi^{(k)}(v)(v_1,\ldots,v_k))
	\|_{\gamma( H, V_\beta) 
	}
}{
\| v_1 \|_V \cdot \ldots \cdot \| v_k \|_V
}	
\\&\leq
\|\iota M\|_{L(\lpnb{\nicefrac{p\delta}{n}}{\lambda_{(0,1)}}{\R},\gamma(H,V_\beta))} 
\sup_{  v_1, \ldots, v_k \in V \backslash \{0\}}
\frac{	
	\|\psi^{(k)}(v)(v_1, \ldots, v_k)\|_{\lpnb{\nicefrac{p\delta}{n}}{\lambda_{(0,1)}}{\R}}
}
{
	\| v_1 \|_V \cdot \ldots \cdot \| v_k \|_V	
}
\\
&
\leq
\|
\iota  M
\|_{L(\lpnb{\nicefrac{p\delta}{n}}{\lambda_{(0,1)}}{\R},\gamma(H,V_\beta))} 
\| \psi^{(k)} ( v ) \|_{L^{(k)}( V, L^{ \nicefrac{p\delta}{n} }( \lambda_{(0,1)}; \R) ) } 
.
\end{split}
\end{equation}
This and~\eqref{eq:77}
ensure that for all
$ k \in \{ 1, \ldots, n \} $
it holds that
\begin{equation}
\begin{split}
\label{eq:114}
& 
\sup_{ v \in V}
\|B^{(k)}(v)\|_{L^{(k)}(V,\gamma(H,V_\beta))}
\leq
\|\iota M\|_{L(\lpnb{\nicefrac{p\delta}{n}}{\lambda_{(0,1)}}{\R},\gamma(H,V_\beta))} 
\bigg[
\sup_{ x \in \R }
| b^{(k)} ( x ) |
\bigg]
.
\end{split}
\end{equation}	
Combining this with~\eqref{eq:Nr9},
\eqref{eq:Nr3f},
and~\eqref{eq:I_estimate}
shows that for all
$ k \in \{1, \ldots, n \} $
it holds that
\begin{equation}
\begin{split}
\label{eq:115}
& 
\sup_{ v \in V }
\|B^{(k)}(v)\|_{L^{(k)}(V,\gamma(H,V_\beta))} 
\\
&
\leq
\| \iota \|_{ \gamma(H_{-\nicefrac{n}{(2p\delta)}},V_\beta) }
\| M \|_{
	L(\lpnb{\nicefrac{p \delta}{n}}{\lambda_{(0,1)}}{\R}, L(H,H_{-\nicefrac{n}{(2p\delta)}}))
}
\bigg[
\sup_{ x \in \R }
| b^{(k)} ( x ) |
\bigg]
\\
&
\leq
\bigg[
\int_\R
\tfrac{ | x |^p }{\sqrt{2\pi}} \, e^{ - \nicefrac{x^2}{2} } 
\, dx
\bigg]^{\nicefrac{1}{p}}
\bigg[
\smallsum\limits_{ l = 1 }^\infty
l^{4(\beta+\nicefrac{n}{(2p\delta)})}  
\bigg]^{\nicefrac{1}{2}}
\\
&
\quad
\cdot
\Bigg[
\sup_{ w \in H_{
		\nicefrac{ n }{ (2 p \delta) }
	} \backslash \{ 0 \} }
\tfrac{
	\| w \|_{
		L^{
			\nicefrac{2 p \delta }{ ( p \delta - 2 n ) }
		}( \lambda_{(0,1)}; \R)
	}
}{
\| w \|_{ H_{
		\nicefrac{ n }{ (2 p \delta) }
	} }
}
\Bigg]
\bigg[
\sup_{ x \in \R }
| b^{(k)} ( x ) |
\bigg]
< \infty
.
\end{split}
\end{equation}
This proves item~\eqref{item:B.Cb}.
Next note that~\eqref{eq:112} demonstrates
that for all
$ k \in \{ 1, \ldots, n \} $,
$ r \in [ \frac{p\delta}{n - k \delta }, \infty ) $, 
$ v, w \in \lpnb{\max\{r,p\}}{\lambda_{(0,1)}}{\R} $, 
$ v_1, \ldots, v_k \in V \setminus \{0\} $ 
it holds that
\begin{equation}
\begin{split}
&
\frac{
	\| ( B^{(k)}(v)
	-
	B^{(k)}(w) ) (v_1, \ldots, v_k )
	\|_{\gamma( H, V_\beta) 
	}
}{
\| v_1 \|_V \cdot \ldots \cdot \| v_k \|_V
}
\\
&
=
\frac{
	\| 
	\iota M( 
	[ \psi^{(k)}(v)
	-
	\psi^{(k)}(w)
	](v_1,\ldots,v_k))
	\|_{\gamma( H, V_\beta) 
	}
}{
\| v_1 \|_V \cdot \ldots \cdot \| v_k \|_V
}	
\\&\leq
\|\iota M\|_{L(\lpnb{\nicefrac{p\delta}{n}}{\lambda_{(0,1)}}{\R},\gamma(H,V_\beta))} \,
\frac{	
	\|
	[ \psi^{(k)}(v)
	-
	\psi^{(k)}(w)
	](v_1, \ldots, v_k)\|_{\lpnb{\nicefrac{p\delta}{n}}{\lambda_{(0,1)}}{\R}}
}
{
	\| v_1 \|_V \cdot \ldots \cdot \| v_k \|_V
}	
\\
&
\leq
\|\iota M\|_{L(\lpnb{\nicefrac{p\delta}{n}}{\lambda_{(0,1)}}{\R},\gamma(H,V_\beta))}
\| \psi^{(k)}(v) - \psi^{(k)}(w)
\|_{L^{(k)}( V, L^{ \nicefrac{p \delta }{n} }( \lambda_{(0,1)}; \R ) ) }
.
\end{split}
\end{equation}
This,
\eqref{eq:Nr3b}, 
and~\eqref{eq:Nr9}
assure that for all
$ k \in \{ 1, \ldots, n\} $,
$ r \in [ \frac{p\delta}{n - k \delta }, \infty ) $ 
it holds that
\begin{equation}
\begin{split}
\label{eq:NrR}
& 
\sup_{\substack{v,w \in \lpnb{\max\{r,p\}}{\lambda_{(0,1)}}{\R},\\ v\neq w}}
\Bigg[
\frac{
	\| B^{(k)}(v)
	-
	B^{(k)}(w)
	\|_{L^{(k)}(V,\gamma( H, V_\beta) )
	}
}{
\|v-w\|_{\lpnb{r}{\lambda_{(0,1)}}{\R}}
}
\Bigg]
\\
&
\leq
\|\iota M\|_{L(\lpnb{\nicefrac{p\delta}{n}}{\lambda_{(0,1)}}{\R},\gamma(H,V_\beta))} 
\Bigg[
\sup_{
	\substack{
		x, y \in \R, \\ x \neq y
	}
} 
\frac{
	| b^{(k)}(x) - b^{(k)}(y) |
}{
| x - y |
}
\Bigg]
\\
&
\leq 
\| \iota \|_{ \gamma(H_{-\nicefrac{n}{(2p\delta)}},V_\beta) }
\| M \|_{
	L(\lpnb{\nicefrac{p \delta}{n}}{\lambda_{(0,1)}}{\R}, L(H,H_{-\nicefrac{n}{(2p\delta)}}))
}
\Bigg[
\sup_{
	\substack{
		x, y \in \R, \\ x \neq y
	}
} 
\frac{
	| b^{(k)}(x) - b^{(k)}(y) |
}{
| x - y |
}
\Bigg]
.
\end{split}
\end{equation}
Combining~\eqref{eq:NrR}
with~\eqref{eq:I_estimate}
and~\eqref{eq:Nr3f}
establishes item~\eqref{item:B:Lip}.
The proof of Proposition~\ref{prop.nonlinearityB} is thus completed.
\end{proof}
\begin{corollary}
	\label{corollary:nonlinearityB}
	Consider the notation in Subsection~\ref{sec:notation}, 
	let $ n \in \N $, $ \beta\in(-\infty,-\nicefrac{1}{4}) $, $ p\in( \max \{ \frac{n+1}{2(|\beta|-\nicefrac{1}{4})}, 2 (n+1) \} ,\infty) $, 
	$
	(V,\left\|\cdot\right\|_V)
	=
	(\lpnb{p}{\lambda_{(0,1)}}{\R},
	$
	$
	\left\|\cdot\right\|_{\lpnb{p}{\lambda_{(0,1)}}{\R}})
	$, 
	let $ b \colon \R \to \R $ be
	an $ n $-times continuously differentiable 
	function
	with globally
	Lipschitz continuous and
	globally
	bounded 
	derivatives, 
	let $ A \colon D(A)\subseteq V \to V $ be the Laplacian with Dirichlet boundary conditions on $ V $, 
	and let 
	$
	(V_r,\left\|\cdot\right\|_{V_r})
	$, 	$ r \in \R $, 
	be a family of interpolation spaces associated to $ -A $.
	Then
	\begin{enumerate}[(i)]
		\item
		\label{item:exist_B}
		there exists a unique continuous function 
		$B\colon V \to \gamma(\lpnb{2}{\lambda_{(0,1)}}{\R},V_\beta)$
		which satisfies for all $v,u\in\lpn{2p}{\lambda_{(0,1)}}{\R}$
		that 
		\begin{equation}
		B\big([v]_{\lambda_{(0,1)},\mathcal{B}(\R)}\big)
		[u]_{\lambda_{(0,1)},\mathcal{B}(\R)}
		=
		\big[\{b(v(x)) \cdot u(x)\}_{x\in(0,1)}\big]_{\lambda_{(0,1)},\mathcal{B}(\R)}
		\end{equation}
		and
		\item
		\label{item:B_smooth}
		it holds that $B$ is $n$-times continuously Fr\'{e}chet differentiable
		with globally
		Lipschitz continuous and globally
		bounded derivatives.
		%
	\end{enumerate}
\end{corollary}
\begin{proof}[Proof of Corollary~\ref{corollary:nonlinearityB}]
	First, note that 
	for all $ k \in \{ 1, 2, \ldots, n \} $ it holds that
	\begin{equation}
	\begin{split} 
	&
	\tfrac{1}{p} \,
	\max \!\big\{
	\tfrac{n}{
		2(|\beta|-\nicefrac{1}{4})
	}
	,
	2n
	\big\}
	=
	\tfrac{n}{p}
	\max \! \big\{  \tfrac{1}{2(|\beta |- \nicefrac{1}{4})}, 2 \big\}
	\\
	&
	=
	\tfrac{n}{p(n+1)}
		\max \! \big\{  \tfrac{n+1}{2(|\beta |- \nicefrac{1}{4})}, 2(n+1) \big\}
	<
	\tfrac{n}{(n+1)}
	< 1
	\end{split}
	\end{equation}
	and 
	\begin{equation}
	\tfrac{
		p(\frac{n}{n+1})
	}{
	n-k(\frac{n}{n+1})
}
=
\tfrac{p}{(n+1)-k}
=
\tfrac{p}{1+n-k}
\leq
p
.
\end{equation}
Items~\eqref{item:exist.B}, \eqref{item:B.smooth},
\eqref{item:finite_constant}, and~\eqref{item:B:Lip}
of
Proposition~\ref{prop.nonlinearityB} 
(with 
$n=n$,
$\beta=\beta$,
$p=p$,
$b=b$,
$\mathcal{A}=A$,
$k=n$,
$\delta=\frac{n}{(n+1)}$,
$r=p$
in the notation of Proposition~\ref{prop.nonlinearityB})
therefore establish items~\eqref{item:exist_B}
and~\eqref{item:B_smooth}.
The proof of Corollary~\ref{corollary:nonlinearityB}
is thus completed.
\end{proof}

\section{Mild stochastic calculus in Banach spaces}
\label{sec:mildcalc}
In this section
we generalize the machinery in~\cite[Section~5]{JentzenKurniawan2015arXiv}
from separable Hilbert spaces to separable UMD Banach spaces
with type 2.
\subsection{Setting}
\label{sec:setting_mild_calculus}
Throughout this section we frequently assume the following setting.
Consider the notation in Subsection~\ref{sec:notation}, 
let $ t_0 \in [ 0, \infty ) $, 
$ T \in ( t_0, \infty ) $, 
$
\angle =
\left\{
(t_1, t_2) \in [ t_0, T ]^2 \colon
t_1 < t_2
\right\}
$,
let 
$
( \Omega, \mathcal{F}, \P )
$
be a probability space with a normal filtration
$ \mathbb{F} = ( \mathbb{F}_t )_{ t \in [ t_0, T ] } $,
let
$
\left( W_t \right)_{ t \in [ t_0, T ] }
$
be an $ \operatorname{Id}_U $-cylindrical
$ ( \Omega, \mathcal{F}, \P, \mathbb{F} ) $-Wiener
process,
let
$
(
\check{V}, 
\left\| \cdot \right\|_{ \check{V} }
)
$,
$
(
V, 
\left\| \cdot \right\|_V
)
$,
$
(
\hat{V}, 
\left\| \cdot \right\|_{ \hat{V} }
)
$, 
and 
$
( \V, 
\left\| \cdot \right\|_V
)
$
be separable
UMD
$  \R $-Banach spaces
with type $ 2 $
which satisfy
$
\check{V}
\subseteq
V
\subseteq
\hat{V}
$
continuously and densely, 
let
$
(
U,
\left< \cdot , \cdot \right>_U,
\left\| \cdot \right\|_U
)
$
be a separable $ \R $-Hilbert space,
let
$
\mathbb{U} \subseteq U
$
be an orthonormal basis of
$ U $,
and for every separable $ \R $-Banach space $ ( E, \left\| \cdot \right\|_E ) $ 
and every
$ a, b \in \R $, $ A \in \mathcal{B}( \R ) $,
$
X \in \mathcal{M}( \mathcal{B}( A ) \otimes \mathcal{F} , \mathcal{B}( E ) )
$
with $ a < b $,
$ (a,b) \subseteq A $,
and
$
\P\big(
\int_a^b \| X_s \|_E \, ds < \infty
\big) = 1
$
let
$
\int_a^b X_s \, {\bf ds}
\in
L^0( \P ; E )
$
be given by
$
\int_a^b X_s \, {\bf ds}
=
\big[
\int_a^b \mathbbm{1}_{ \{ \int_a^b \| X_u \|_E \, du < \infty \} } X_s \, ds
\big]_{
	\P , \mathcal{B}( E )
}
$.
\subsection{Mild It\^{o} processes}
\label{sec:mildprocesses}
\begin{definition}[Mild It\^{o} process]
	\label{def:mildIto}
	Consider the notation in Subsection~\ref{sec:notation},
	let 
	$
	(
	\check{V}, 
	\left\| \cdot \right\|_{ \check{V} }
	)
	$,
	$
	(
	V, 
	\left\| \cdot \right\|_{ V }
	)
	$,
	and
	$
	(
	\hat{V}, 
	\left\| \cdot \right\|_{ \hat{V} }
	)
	$ 
	be separable
	UMD
	$  \R $-Banach spaces
	with type $ 2 $
	which satisfy
	$
	\check{V}
	\subseteq
	V
	\subseteq
	\hat{V}
	$
	continuously and densely, 
	let
	$
	(
	U,
	\left< \cdot , \cdot \right>_U,
	\left\| \cdot \right\|_U
	)
	$
	be a separable $ \R $-Hilbert space,
	let $ t_0 \in [0,\infty) $, $ T \in (t_0, \infty) $, 
	let 
	$
	( \Omega, \mathcal{F}, \P )
	$
	be a probability space with a normal filtration
	$ \mathbb{F} = ( \mathbb{F}_t )_{ t \in [ t_0, T ] } $,
	and
	let
	$
	\left( W_t \right)_{ t \in [ t_0, T ] }
	$
	be an $ \operatorname{Id}_U $-cylindrical
	$ ( \Omega, \mathcal{F}, \P, \mathbb{F} ) $-Wiener
	process.
	Then we say that $ X $ is a mild It\^o process on
	$ ( \Omega, \mathcal{F}, \P, \mathbb{F}, W, 
	( \check{V}, \left \| \cdot \right \|_{\check{V}} ) $,
	$ ( V, \left \| \cdot \right \|_V ), 
	( \hat{V}, \left \| \cdot \right \|_{ \hat{V} } ) ) $
	with
	evolution family $ S $, mild drift $ Y $, and mild diffusion $ Z $
	(we say that $ X $ is a mild It\^o process 
	with
	evolution family $ S $, mild drift $ Y $, and mild diffusion $ Z $,
	we say that
	$ X $ is a mild It\^o process)
	if and only if it holds
	\begin{enumerate}[(i)]
		\item  that 
		$ X \in \mathbb{M}( [t_0, T ] \times \Omega, V ) $ 
		is an $ \mathbb{F} / \mathcal{B}( V ) $-predictable stochastic process,
		\item that 
		$ Y \in \mathbb{M}( [t_0, T ] \times \Omega, \hat{V} ) $ 
		is an $ \mathbb{F} / \mathcal{B}( \hat{V} ) $-predictable stochastic process,
		\item that 
		$ Z \in \mathbb{M}( [t_0, T] \times \Omega, \gamma (U, \hat{V}) ) $ 
		is an $ \mathbb{F} / \mathcal{B}( \gamma( U, \hat{V} ) ) $-predictable stochastic process,
		\item that 
		$
		S \in \mathbb{M}(
		\left\{
				(t_1, t_2) \in [ t_0, T ]^2 \colon
				t_1 < t_2
				\right\}, 
		L( \hat{V}, \check{V} )
		)
		$
		is a
		$
		\mathcal{B}( \left\{
		(t_1, t_2) \in [ t_0, T ]^2 \colon
		t_1 < t_2
		\right\} )
		/ $
		$
		\mathcal{S}( \hat{V}, \check{V} )
		$-measurable
		function which satisfies for all $ t_1, t_2, t_3 \in [ t_0, T ] $
		with $ t_1 < t_2 < t_3 $
		that
		$
		S_{ t_2, t_3 }
		S_{ t_1, t_2 }
		=
		S_{ t_1, t_3 }
		$,
		\item that $ \forall \, t \in (t_0, T] \colon \P ( \int_{ t_0 }^{ t }
		\|
		S_{ s, t } Y_s
		\|_{ \check{V} }
		+
		\|
		S_{ s, t } Z_s
		\|_{ \gamma( U, \check{V} ) }^2
		\,
		ds < \infty ) = 1 $, and
		\item that for all $ t \in (t_0,T] $
		it holds that
		\begin{equation}
		\label{eq:mildito}
		\!\!\!
		[ X_t ]_{\P, \mathcal{B}( V ) }
		=
		\bigg[
		S_{ t_0, t }   
		X_{ t_0 } 
		+ 
		\int_{ t_0 }^t
		\mathbbm{1}_{
			\{
			\int_{ t_0 }^t
			\|	S_{ s, t } \,
			Y_s \|_V
			\, ds
			< \infty 
			\}
		}
		S_{ s, t }  
		Y_s
		\, ds
		\bigg]_{ \P, \B(V) }
		+
		\int_{ t_0 }^t
		S_{ s, t }  
		Z_s
		\, dW_s
		.
		\end{equation}
	\end{enumerate}
\end{definition}

\begin{lemma}[Regularization of mild It\^{o} processes]
	\label{lemma:transformation}
	Assume the setting in Subsection~\ref{sec:setting_mild_calculus}
	and let
	$
	X \colon [ t_0, T ] \times \Omega
	\rightarrow V
	$
	be a mild It{\^o} process
	with evolution family
	$
	S \colon \angle
	\rightarrow L( \hat{V}, \check{V} )
	$,
	mild drift
	$
	Y \colon [ t_0, T ] \times
	\Omega \rightarrow \hat{V}
	$, 
	and mild
	diffusion
	$
	Z \colon [ t_0, T ] \times
	\Omega \rightarrow
	\gamma(U,\hat{V})
	$.
	Then  
	there exists an up to indistinguishability unique
	stochastic process 
	$
	\bar{X} \colon [ t_0, T ] \times \Omega \rightarrow \check{V}
	$ 
	with continuous sample paths
	which satisfies 
	$
	\forall \, 
	t \in [ t_0, T ) 
	\colon
	\P\big(
	\bar{X}_t
	=
	S_{ t, T } X_t
	\big) = 1
	$. 
\end{lemma}

\begin{proof}[Proof of Lemma~\ref{lemma:transformation}]
	The assumption that $ X $ is a mild It\^{o} process, in particular, 
	ensures that  
	$
	\P ( \int_{ t_0 }^T
	\left\| S_{ s, T } Y_s \right\|_{
		\check{V}
	}
	+
	\left\| S_{ s, T } Z_s \right\|_{
		\gamma( U, \check{V} )
	}^2
	ds < \infty ) = 1
	$.
	This implies that 
	there exists a stochastic 
	process 
	$ 
	\bar{X} \colon [ t_0, T ] \times \Omega \to \check{V}
	$
	with continuous sample paths
	which satisfies for all 
	$ t \in [ t_0, T ] $ that
	\begin{equation}
	\label{eq:semiX}
	[ \bar{X}_t ]_{\P, \mathcal{B}( \check{V}) }
	=
	[ S_{ t_0, T } \, X_{ t_0 }  ]_{\P, \mathcal{B}( \check{V}) }
	+
	\int_{ t_0 }^t
	S_{ s, T } \, Y_s \, { \bf ds }
	+
	\int_{ t_0 }^t
	S_{ s, T } \, Z_s \, dW_s
	.
	\end{equation}
	Next observe that Definition~\ref{def:mildIto}
	ensures  
	for all $ t \in ( t_0, T ) $ that
	\begin{equation}
	\begin{split}
	&
	[ S_{ t_0, T } \, X_{ t_0 }  ]_{\P, \mathcal{B}( \check{V}) }
	+
	\int_{ t_0 }^t
	S_{ s, T } \, Y_s \, { \bf ds }
	+
	\int_{ t_0 }^t
	S_{ s, T } \, Z_s \, dW_s
	\\ & =
	S_{ t, T }
	\left(
	[ S_{ t_0, t } \, X_{ t_0 }  ]_{\P, \mathcal{B}( \check{V}) }
	+
	\int_{ t_0 }^t
	S_{ s, t } \,
	Y_s \, { \bf ds }
	+
	\int_{ t_0 }^t
	S_{ s, t } \,
	Z_s \, dW_s
	\right)
	=
	S_{ t, T } \,
	[ X_t  ]_{\P, \mathcal{B}( \check{V}) }
	.
	\end{split}
	\end{equation}
	Hence, we obtain
	for all $ t \in [ t_0, T ) $
	that
	\begin{equation}
	\label{eq:fact}
	\begin{split}
	&
	[ S_{ t_0, T } \, X_{ t_0 }  ]_{\P, \mathcal{B}( \check{V}) }
	+
	\int_{ t_0 }^t
	S_{ s, T } \, Y_s \, { \bf ds }
	+
	\int_{ t_0 }^t
	S_{ s, T } \, Z_s \, dW_s
	=
	[ S_{ t, T } \,
	X_t  ]_{\P, \mathcal{B}( \check{V}) }
	.
	\end{split}
	\end{equation}
	Combining this and~\eqref{eq:semiX}
	shows that for all
	$ t \in [ t_0, T ) $ it holds that
	\begin{equation}
	\label{eq:semiX_2}
	[ \bar{X}_t  ]_{\P, \mathcal{B}( \check{V}) } 
	=
	[ S_{ t_0, T } \, X_{ t_0 }  ]_{\P, \mathcal{B}( \check{V}) }
	+
	\int_{ t_0 }^t
	S_{ s, T } \, Y_s \, { \bf ds }
	+
	\int_{ t_0 }^t
	S_{ s, T } \, Z_s \, dW_s
	= 
	[ S_{ t, T } X_t  ]_{\P, \mathcal{B}( \check{V}) } 
	.
	\end{equation}
	Moreover, observe that for all 
	stochastic processes 
	$ A, B \colon [0,T] \times \Omega \to \check{V} $
	with continuous sample paths
	which satisfy
	$
	\forall \, t \in [t_0,T)
	\colon
	\P\big(
	A_t = B_t
	\big) = 1
	$
	it holds that
	$
	\P\big(
	\forall \, t \in [ t_0, T ] \colon
	A_t = B_t
	\big)
	= 1
	$.
	Combining this with \eqref{eq:semiX_2} 
	completes the proof of 
	Lemma~\ref{lemma:transformation}.
\end{proof}
\begin{lemma}[Regularization of mild It\^{o} processes]
	\label{lemma:transformation2}
	Assume the setting in Subsection~\ref{sec:setting_mild_calculus}, let
	$
	X \colon [ t_0, T ] \times \Omega
	\rightarrow V
	$
	be a mild It{\^o} process
	with evolution family
	$
	S \colon \angle
	\rightarrow L( \hat{V}, \check{V} )
	$,
	mild drift
	$
	Y \colon [ t_0, T ] \times
	\Omega \rightarrow \hat{V}
	$, 
	and mild
	diffusion
	$
	Z \colon [ t_0, T ] \times
	\Omega \rightarrow
	\gamma(U,\hat{V})
	$,
	and let 
	$
	\bar{X} \colon [ t_0, T ] \times \Omega \rightarrow \check{V}
	$ 
	be a stochastic process
	with continuous sample paths
	which satisfies
	$ 
	\forall \, t \in [ t_0, T ) \colon
	\P\big(
	\bar{X}_t
	=
	S_{ t, T } X_t
	\big) = 1
	$.
	Then  
	\begin{enumerate}[(i)]
		\item
		\label{item:transformation_i} 
		it holds
		that
		$ \bar{X} $ 
		is 
		$ \mathbb{F}/\B( \check{V} ) $-predictable,
		\item \label{item:transfortmation2}
		it holds 
		that
		$ \P( 
		\bar{X}_T = X_T
		) = 1
		$,
		\item \label{item:well_defined} 
		it holds
		that
		$
		\P ( \int_{ t_0 }^T
		\left\| S_{ s, T } Y_s \right\|_{
			\check{V}
		}
		+
		\left\| S_{ s, T } Z_s \right\|_{
			\gamma( U, \check{V} )
		}^2
		ds < \infty ) = 1
		$,
		and
		\item \label{eq:SomeIto}
		it holds
		that
		\begin{equation}
		\forall \, t \in [t_0,T] \colon
		[ \bar{X}_t ]_{\P, \mathcal{B}( \check{V}) }
		=
		[
		S_{ t_0, T } \,
		X_{ t_0 }
		]_{\P, \mathcal{B}( \check{V}) }
		+
		\int_{ t_0 }^t
		S_{ s, T } \,
		Y_s
		\, { \bf ds }
		+
		\int_{ t_0 }^t
		S_{ s, T } \,
		Z_s
		\, dW_s
		.
		\end{equation}
	\end{enumerate}
\end{lemma}
\begin{proof}[Proof of Lemma~\ref{lemma:transformation2}]
	The assumption that $ \bar X $ has continuous sample paths,
	the fact that
	$ X $ is $ \mathbb{F}/ \B(V) $-adapted,
	and
	the fact that
	$ \forall \, t \in [t_0, T) \colon 
	\P( \bar X_t = S_{t,T} X_t ) = 1 $
	establish item~\eqref{item:transformation_i}.
	Moreover, note that the assumption that $ X $ is a mild It\^{o} process 
	proves item~\eqref{item:well_defined}.
	In addition, observe that the assumption that
	$ 
	\forall \, t \in [ t_0, T ) \colon
	\P\big(
	\bar{X}_t
	=
	S_{ t, T } X_t
	\big) = 1
	$
	implies that
	for all
	$ t \in [ t_0, T ) $
	it holds 
	that
	\begin{equation}
	[ \bar{X}_t  ]_{\P, \mathcal{B}( \check{V}) } 
	= 
	[ S_{ t, T } X_t  ]_{\P, \mathcal{B}( \check{V}) } 
	=
	[ S_{ t_0, T } \, X_{ t_0 }  ]_{\P, \mathcal{B}( \check{V}) }
	+
	\int_{ t_0 }^t
	S_{ s, T } \, Y_s \, { \bf ds }
	+
	\int_{ t_0 }^t
	S_{ s, T } \, Z_s \, dW_s
	.
	\end{equation} 
	Combining this with the assumption that $ \bar X $ has continuous sample paths
	shows items~\eqref{eq:SomeIto}
	and~\eqref{item:transfortmation2}.
	The proof of Lemma~\ref{lemma:transformation2}
	is thus completed.
\end{proof}
\subsection{Standard It\^o formula}
\label{sec:standardIto}
Theorem~\ref{theorem:standardIto}
is an elementary extension
of 
Theorem~2.4 
in Brze\'zniak et al.\ \cite{BrzezniakVanNeervenVeraarWeis2008}
(cf.\ Lemma~\ref{lemma:gamma_estimate} 
in Subsection~\ref{section:Preliminary} above).
\begin{theorem}
	\label{theorem:standardIto}
	Assume the setting in Subsection~\ref{sec:setting_mild_calculus},    
	let
	$ \varphi 
	= ( \varphi(t,x))_{t \in [t_0,T], x \in V} \in C^{1,2}( [t_0,T] \times V, \V ) $,
	$ \xi \in \mathcal{M}( \mathbb{F}_{t_0}, \mathcal{B}(V) ) $,
	let
	$ Z \colon [t_0,T] \times \Omega \to \gamma( U, V ) $
	be an  
	$ \mathbb{F} / \B( \gamma(U, V)) $-predictable stochastic process
	which satisfies
	$ \P(
	\int_{t_0}^T \| Z_t \|_{\gamma(U, V)}^2
	\, dt
	< \infty
	) = 1 $,
	let
	$ Y \colon [t_0,T] \times \Omega \to V $
	be 
	an 
	$ \mathbb{F}/\B(V) $-predictable stochastic process
	which satisfies
	$ \P ( \int_{t_0}^T \| Y_t \|_V \, dt < \infty ) = 1 $,
	and let
	$ X \colon [t_0,T] \times \Omega \to V $
	be an
	$ \mathbb{F}/\B(V) $-predictable stochastic process
	which satisfies for all
	$ t \in [t_0,T] $ that
	\begin{equation}
	[ X_t ]_{\P, \B(V) }
	=
	[ \xi ]_{\P, \B(V) }  
	+
	\int_{t_0}^t
	Y_s \, { \bf ds } 
	+
	\int_{t_0}^t
	Z_s \, dW_s.
	\end{equation}
	Then 
	\begin{enumerate}[(i)]
		\item it holds that
		$ \P( \int_{t_0}^T 
		\| (\tfrac{\partial}{\partial t} \varphi)(s, X_s) \|_\V \, ds
		< \infty 
		) = 1 $,
		\item 
		it holds that
		$ \P( \int_{t_0}^T 
		\| ( \tfrac{\partial}{\partial x} \varphi )(s, X_s) Y_s \|_\V \, ds
		< \infty 
		) = 1 $, 
		\item it holds that 
		$ \P ( \int_{t_0}^T \| (\tfrac{\partial}{\partial x} \varphi)(s, X_s) Z_s
		\|_{\gamma(U, \V)}^2 \, ds < \infty ) = 1 $,
		\item
		it holds for
		all $ \omega \in \Omega $,
		$ s \in [t_0,T] $
		that there exists a unique $ v \in \V $ such that
		\begin{equation}
		\sup_{ \substack{ I \subseteq \mathbb{U}, \\ \#_I < \infty } }
		\sup_{ \substack{ I \subseteq J \subseteq \mathbb{U}, \\
				\#_J < \infty } }
		\bigg\|
		v
		-
		\smallsum\limits_{u \in J } 
		(\tfrac{ \partial^2 }{ \partial x^2} \varphi )
		(s, X_s(\omega) ) ( Z_s(\omega) u, Z_s(\omega) u )
		\bigg\|_\V
		=
		0,
		\end{equation}
		\item
		it holds that
		\begin{equation}
		\begin{split}
		& 
		\P\Bigg(
		\int_{t_0}^T
		\bigg\| 
		\smallsum\limits_{u \in \mathbb{U} }
		(\tfrac{ \partial^2 }{ \partial x^2} \varphi )(s, X_s  ) ( Z_s u, Z_s  u )
		\bigg\|_\V
		\, ds
		< \infty 
		\Bigg) = 1
		,
		\end{split}
		\end{equation}
		and
		\item 
		it holds for all
		$ t_1 \in [t_0,T] $ that
		\begin{equation}
		\begin{split}
		& 
		[ \varphi( t_1, X_{t_1} ) -\varphi(t_0, X_{t_0}) ]_{\P, \B( \V )}
		=
		\int_{t_0}^{t_1}	 
		\big[
		( \tfrac{\partial}{\partial t} \varphi )(s, X_s) \, + 
		( \tfrac{\partial}{\partial x} \varphi )(s, X_s) Y_s
		\big] \, { \bf ds } 
		\\
		& 
		+ 
		\tfrac{1}{2}
		\int_{t_0}^{t_1}
		\smallsum\limits_{u \in \mathbb{U} }
		(\tfrac{ \partial^2 }{ \partial x^2} \varphi )(s, X_s)  
		(Z_s u, Z_s u )
		\, { \bf ds }
		+
		\displaystyle\int_{t_0}^{t_1} 
		(\tfrac{\partial}{\partial x } \varphi )(s, X_s) Z_s \, dW_s
.
		\end{split}
		\end{equation}
	\end{enumerate}
\end{theorem}

\subsection{Mild It{\^o} formula for stopping times}
\label{sec:mildito}

\begin{theorem}[Mild It{\^o} formula]
	\label{thm:ito}
	Assume the setting in Subsection~\ref{sec:setting_mild_calculus},
	let
	$
	X \colon [ t_0, T ] \times \Omega
	\rightarrow V
	$
	be a mild It{\^o} process
	with evolution family
	$
	S \colon \angle
	\rightarrow L( \hat{V}, \check{V} )
	$,
	mild drift
	$
	Y \colon [ t_0, T ] \times
	\Omega \rightarrow \hat{V}
	$, 
	and mild
	diffusion
	$
	Z \colon [ t_0, T ] \times
	\Omega \rightarrow
	\gamma(U,\hat{V})
	$,
	let
	$
	\bar{X} \colon [ t_0, T ] \times \Omega \rightarrow \check{V}
	$ 
	be a
	stochastic process with continuous sample paths
	which satisfies
	$
	\forall \, t \in [ t_0, T ) \colon
	\P (
	\bar{X}_t
	=
	S_{ t, T } 
	X_t
	) = 1
	$
	(see Lemma~\ref{lemma:transformation}),
	let
	$ r \in [ t_0, T ) $,
	$
	\varphi
	=
	(
	\varphi(t,x)
	)_{ t \in [ r, T ],\, x \in \check{V} }
	\in C^{1,2}(
	[ r, T ] \times \check{V}, \V
	)
	$,
	and let
	$ \tau \colon \Omega \rightarrow [ r, T ] $
	be an 
	$
	\mathbb{F}
	$-stopping time.
	Then 
	\begin{enumerate}[(i)]
		\item \label{item:first}
		it holds that
		$ \P (
		\int_r^{ T }
		\|
		(
		\tfrac{ \partial }{ \partial x }
		\varphi
		)
		( s, S_{ s, T } X_s )
		S_{ s, T } Y_s
		\|_\V
		\, ds
		< \infty)= 1 $,
		\item it holds that
		$ \P( \int_r^T
		\|
		(
		\tfrac{ \partial }{ \partial x }
		\varphi
		)
		( s, S_{ s, T } X_s )
		S_{ s, T } Z_s
		\|_{ \gamma(U, \V ) }^2
		\,
		ds
		< \infty
		) = 1 $, 
		\item
		it holds that
		$ \P ( 
		\int_r^{ T }
		\|
		(
		\tfrac{ \partial }{ \partial t }
		\varphi
		)
		( s, S_{ s, T } X_s )
		\|_\V
		\, ds < \infty ) = 1 $,
		\item \label{item:fourth}
		it holds that
		$ \P( \int_r^T
		\|
		(
		\tfrac{ \partial^2 }{ \partial x^2 }
		\varphi
		)
		( s, S_{ s, T } X_s )
		\|_{ L^{(2)}( \check{V}, \V ) } \,
		\|
		S_{ s, T } Z_s
		\|_{ \gamma(U, \check{V} ) }^2 \,
		ds
		< \infty
		) = 1 $,
		\item \label{item:well_defined_sum}
		it holds for
		all $ \omega \in \Omega $,
		$ s \in [r,T] $
		that there exists a unique $ v \in \V $ such that
		\begin{equation}
		\sup_{ \substack{ I \subseteq \mathbb{U}, \\ \#_I < \infty } }
		\sup_{ \substack{ I \subseteq J \subseteq \mathbb{U}, \\
				\#_J < \infty } }
		\bigg\|
		v
		-
		\smallsum\limits_{u \in J} 
		(\tfrac{ \partial^2 }{ \partial x^2} 
		\varphi )
		(s, S_{s,T} X_s(\omega) ) 
		( S_{s,T} Z_s(\omega) u,  
		S_{s,T} Z_s(\omega) u )
		\bigg\|_\V
		=
		0,
		\end{equation}
		\item \label{item:finite}
		it holds that
		\begin{equation}
		\begin{split}
		& 
		\P\bigg(
		\int_{t_0}^T
		\bigg\| 
		\smallsum\limits_{u \in \mathbb{U} } 
		(\tfrac{ \partial^2 }{ \partial x^2} 
		\varphi )
		(s, S_{s,T} X_s  )   
		( S_{s,T} Z_s  u, 
		S_{s,T} Z_s  u )
		\bigg\|_\V
		\, ds
		< \infty 
		\bigg)
		=
		1
		,
		\end{split}
		\end{equation}
		and
		\item
		\label{item:Mild_Ito}
		it holds that
		\begin{equation}
		\label{eq:itoformel_start}
		\begin{split}
		[ \varphi( \tau, \bar{X}_\tau ) ]_{\P, \mathcal{B}( \V ) }
		&
		=
		[ \varphi( r,
		S_{ r, T }
		X_{ r }
		) ]_{\P, \mathcal{B}( \V ) }
		+
		\int_{ r }^\tau
		(
		\tfrac{ \partial }{ \partial x } \varphi
		)
		( s, S_{ s, T } X_s )
		\, S_{ s, T } \, Z_s
		\, dW_s
		\\
		&
		\quad
		+
		\int_{ r }^\tau
		\big[
		(
		\tfrac{ \partial }{ \partial t }
		\varphi
		)
		( s, S_{ s, T } X_s )
		+
		(
		\tfrac{ \partial }{ \partial x } \varphi
		)
		( s, S_{ s, T } X_s )
		\,
		S_{ s, T} \,Y_s
		\big]
		\,{ \bf ds} 
		\\
		&
		\quad
		+
		\tfrac{1}{2}
		\int_{ r }^\tau
		\smallsum\limits_{ u \in \mathbb{U} }
		(
		\tfrac{ \partial^2 }{ \partial x^2 } \varphi
		)
		( s, S_{ s, T } X_s )
	 (
		S_{ s, T }
		Z_s u ,
		S_{ s, T }
		Z_s u
	 ) \, { \bf ds}
		.
		\end{split}
		\end{equation}
	\end{enumerate}
\end{theorem}
\begin{proof}[Proof
	of Theorem~\ref{thm:ito}]
	Throughout this proof let 
	$
	\varphi_{1,0} \colon [ r, T ] \times \check{V} \to \V 
	$, 
	$
	\varphi_{0,1} \colon [ r, T ] \times \check{V} \to L( \check{V}, \V ) 
	$, 
	and
	$
	\varphi_{0,2} \colon [ r, T ] \times \check{V} \to L^{(2)}( \check{V}, \V ) 
	$
	be the functions which satisfy
	for all 
	$ t \in [ r, T ] $, 
	$ x, v_1, v_2 \in \check{V} $ 
	that 
	$
	\varphi_{1,0}( t, x )
	=
	\big(\tfrac{ \partial }{ \partial t } \varphi\big)( t, x)
	,
	$
	$
	\varphi_{0,1}( t, x ) \, v_1
	=
	\big(\tfrac{ \partial }{ \partial x } \varphi\big)( t, x ) \, v_1
	,
	$
	and
	$
	\varphi_{0,2}( t, x )( v_1, v_2 )
	=
	\big(
	\tfrac{ \partial^2 }{ \partial x^2 } \varphi
	\big)( t, x )( v_1, v_2 )
	$.
	Note that Lemma~\ref{lemma:transformation2}
	ensures that
	$ \bar{X} $ is an $ \mathbb{F}/\B(\check{V}) $-adapted
	stochastic process with continuous sample paths
	which satisfies for all
	$ t \in [ t_0, T ] $ that
	\begin{equation}
	[ \bar{X}_t ]_{\P, \mathcal{B}( \check{V}) }
	=
	[
	S_{ t_0, T } \,
	X_{ t_0 }
	]_{\P, \mathcal{B}( \check{V}) }
	+
	\int_{ t_0 }^t
	S_{ s, T } \,
	Y_s
	\, { \bf ds }
	+
	\int_{ t_0 }^t
	S_{ s, T } \,
	Z_s
	\, dW_s
	.
	\end{equation}
	Moreover, the assumption that
	$ \varphi \in C^{1,2}([r,T]\times 
	\check{V}, \V ) $,
	the assumption that 
	$
	\bar{X}
	\colon [t_0, T]
	\times \Omega
	\rightarrow \check{V}
	$
	has continuous sample paths,
	and 
	the fact that
	$ 
	\forall\,  t \in [t_0, T] \colon
	\P ( \int_{ t_0 }^{ t }
	\|
	S_{ s, t } Y_s
	\|_{ \check{V} }
	+
	\|
	S_{ s, t } Z_s
	\|_{ \gamma( U, \check{V} ) }^2
	\,
	ds < \infty ) = 1 $ 
	imply
	that
	\begin{equation}
	\label{eq:well1c}
	\mathbb{P}\!\left(
	\int_{ r }^T
	\|
	\varphi_{0,1} ( s, \bar{X}_s )
	S_{ s, T } Y_s
	\|_\V
	+ 
	\|
	\varphi_{0,1} ( s, \bar{X}_s )
	S_{ s, T } Z_s
	\|_{ \gamma(U, \V ) }^2 
	\,
	ds
	< \infty 
	\right) = 1
	\end{equation}
	and
	\begin{equation}
	\label{eq:well2c}
	\mathbb{P}\!\left(
	\int_{ \tau }^T
	\|
	\varphi_{1,0}( s, \bar{X}_s )
	\|_\V
	+
	\|
	\varphi_{0,2}( s, \bar{X}_s )
	\|_{ L^{(2)}( \check{V}, \V ) }
	\|
	S_{ s, T } Z_s
	\|_{ \gamma(U, \check{V} ) }^2 
	\,
	ds
	< \infty 
	\right) = 1 
	.
	\end{equation}
	Combining this with, e.g.,
	Lemma~3.1
	in~\cite{JentzenPusnik2016}
	proves
	items~\eqref{item:first}--\eqref{item:fourth}.
	%
	%
	%
	Then note that Lemma~\ref{lemma:transformation2}
	and 
	Theorem~\ref{theorem:standardIto}
	show 
	\begin{enumerate}[(a)]
		\item \label{item:aa}
	 that for all 
		$ \omega \in \Omega $,
		$ s \in [r,T] $
	 there exists a unique $ v \in \V $ such that
		\begin{equation}
		\sup_{ \substack{ I \subseteq \mathbb{U}, \\ \#_I < \infty } }
		\sup_{ \substack{ I \subseteq J \subseteq \mathbb{U}, \\
				\#_J < \infty } }
		\bigg\|
		v
		-
		\sum_{h \in J } 
		\varphi_{0,2}(s, \bar X_s(\omega) ) 
		( S_{s,T} Z_s(\omega) u, S_{s,T} Z_s(\omega) u )
		\bigg\|_\V
		=
		0,
		\end{equation}
		\item \label{item:bb}
	  that
		\begin{equation}
		\begin{split}
		&
		\P\bigg(
		\int_0^T
		\bigg\| 
		\sum_{u \in \mathbb{U} }
		\varphi_{0,2}(s, \bar X_s  )
		( S_{s,T}Z_s  u, S_{s,T}Z_s  u)
		\bigg\|_\V
		\, ds
		< \infty\bigg) = 1
		,
		\end{split}
		\end{equation}
		and
		\item \label{item:cc}
	    that
		\begin{equation}
		\label{eq:itoformel2inProof}
		\begin{split}
		&
		[ \varphi( \tau, \bar{X}_{ \tau } )  ]_{\P, \mathcal{B}( \V ) }
		= 
		[ \varphi( r, \bar{X}_{ r } ) ]_{\P, \mathcal{B}( \V ) }
		+
		\int_{r}^{ \tau }
		\varphi_{1,0}( s, \bar{X}_{ s } )
		+
		\varphi_{0,1}( s, \bar{X}_{ s } ) 
		S_{ s, T }  Y_s \, { \bf ds}
		\\&+
		\int_{r}^{ \tau }
		\varphi_{0,1}( s, \bar{X}_{ s } ) 
		S_{ s, T }  Z_s \, dW_s
		+
		\tfrac{1}{2}
		\int_{r}^{ \tau }
		\sum_{ u \in \mathbb{U} }
		\varphi_{0,2}( s,
		\bar{X}_{ s }
		)
		\left(
		S_{ s, T } Z_s  u ,
		S_{ s, T } Z_s  u
		\right) { \bf ds}
		.
		\end{split}
		\end{equation}
	\end{enumerate}
	Combining this
	with, e.g.,
	Lemma~3.1
	in~\cite{JentzenPusnik2016},
	the fact that 
	$ 
	\forall \, t \in [ t_0, T ) 
	\colon
	\P\big(
	\bar{X}_t = S_{ t, T } \, X_t
	\big)
	= 1
	$,
	and the fact that
	$ \forall \, t \in [t_0, T] 
	\colon
	\P  ( 
	\sum_{ u \in \mathbb{U} }
	\varphi_{0,2}( s,
	\bar{X}_{ s }
	)
	(
	S_{ s, T } Z_s u ,
	S_{ s, T } Z_s u
	)
	=
	\sum_{ u \in \mathbb{U} }
	\varphi_{0,2}( s,
	S_{ s, T }
	X_s
	)
	(
	S_{ s, T }
	Z_s u ,
	S_{ s, T }
	Z_s u
	)
	) = 1 
	$
	shows that
	item~\eqref{item:well_defined_sum} holds, 
	that item~\eqref{item:finite} holds,
	and that for all
	$ t \in [r,T] $
	it holds that
	\begin{equation}
	\label{eq:itoformel_fixed}
	\begin{split}
	[ \varphi( t , \bar{X}_t ) ]_{\P, \mathcal{B}( \V ) }
	&
	=
	[
	\varphi( r,
	S_{ r, T }
	X_{ r }
	)
	]_{\P, \mathcal{B}( \V ) }
	+
	\int_{ r }^t
	\varphi_{0,1}( s,
	S_{ s, T }
	X_s
	) \,
	S_{ s, T } \,
	Z_s \, dW_s
	\\
	&
	\quad
	+
	\int_{ r }^t
	\varphi_{1,0}( s,
	S_{ s, T }
	X_s
	)
	+
	\varphi_{0,1}( s,
	S_{ s, T }
	X_s
	) \,
	S_{ s, T} \,
	Y_s \, { \bf ds}
	\\
	&
	\quad
	+
	\tfrac{1}{2}
	\int_{ r }^t
	\sum_{ u \in \mathbb{U} }
	\varphi_{0,2}( s,
	S_{ s, T }
	X_s
	)
	\left(
	S_{ s, T }
	Z_s u ,
	S_{ s, T }
	Z_s u
	\right) { \bf ds}
	.
	\end{split}
	\end{equation}
	This implies item~\eqref{item:Mild_Ito}.
	The proof of Theorem~\ref{thm:ito}
	is thus completed.
\end{proof}

\begin{definition}[Extended mild Kolmogorov operators]
	Assume the setting in Subsection~\ref{sec:setting_mild_calculus}, 
	let 
	$
	S \colon
	\angle
	\rightarrow
	L( \hat{V}, \check{V} )
	$
	be a
	$
	\mathcal{B}( \angle )
	$/$
	\mathcal{S}( \hat{V}, \check{V} )
	$-measurable
	function which satisfies for all $ t_1, t_2, t_3 \in [ t_0, T ] $
	with $ t_1 < t_2 < t_3 $
	that
	$
	S_{ t_2, t_3 }
	S_{ t_1, t_2 }
	=
	S_{ t_1, t_3 }
	$, 
	and let 
	$ (t_1, t_2) \in \angle $.
	Then we denote by 
	$
	\mathcal{L}^S_{ t_1, t_2 }
	\colon
	C^2( \check{V}, \V )
	\rightarrow
	C( V \times \hat{V} \times \gamma( U, \hat{V} ), \V )
	$ 
	the function which satisfies for all 
	$ \varphi \in C^2( \check{V}, \V ) $, 
	$ x \in V $, 
	$ y \in \hat{V} $, 
	$ z \in \gamma( U, \hat{V} ) $ 
	that 
	\begin{equation}
	\big(
	\mathcal{L}^S_{ t_1, t_2 } \varphi
	\big)( x, y, z )
	=
	\varphi'( S_{ t_1, t_2 } \, x )
	\, S_{ t_1, t_2 } \, y
	+
	\tfrac{ 1 }{ 2 }
	\sum_{ u \in \mathbb{U} }
	\varphi''( S_{ t_1, t_2 } \, x )
	( S_{ t_1, t_2 }   z   u, S_{ t_1, t_2 }  z   u )
	.
	\end{equation}
\end{definition}
The next corollary of 
Theorem~\ref{thm:ito} specialises Theorem~\ref{thm:ito}
to the case 
where
$ r = t_0 $
and where 
the test function
$
( \varphi(t, x) )_{
	t \in [ t_0, T ], \,
	x \in \check{V}
}
\in
C^{1,2}(
[ t_0, T ] \times \check{V}, \V
)
$
depends on $ x \in \check{V} $ only.
\begin{corollary}
	\label{cor:itoauto}
	Assume the setting in Subsection~\ref{sec:setting_mild_calculus},
	let
	$
	X \colon [ t_0, T ] \times \Omega
	\rightarrow V
	$
	be a mild It{\^o} process
	with evolution family
	$
	S \colon \angle
	\rightarrow L( \hat{V}, \check{V} )
	$,
	mild drift
	$
	Y \colon [ t_0, T ] \times
	\Omega \rightarrow \hat{V}
	$, 
	and mild
	diffusion
	$
	Z \colon [ t_0, T ] \times
	\Omega \rightarrow
	\gamma(U,\hat{V})
	$, 
	let
	$
	\bar{X} \colon [ t_0, T ] \times \Omega \rightarrow \check{V}
	$ 
	be a stochastic process
	with continuous sample paths
	which satisfies
	$
	\forall \, t \in [ t_0, T ) \colon
	\P\big(
	\bar{X}_t
	=
	S_{ t, T } X_t
	\big) = 1
	$
	(see Lemma~\ref{lemma:transformation}),
	let
	$
	\varphi \in C^2(
	\check{V}, \V
	)
	$, 
	and let 
	$ \tau \colon \Omega \rightarrow [ t_0, T ] $ 
	be an
	$ \mathbb{F} $-stopping time.
	Then  
	\begin{enumerate}[(i)] 
\item
it holds that
$ \P( 
	\int_{ t_0 }^T
	\|
		(
		\mathcal{L}^S_{ s, T } \varphi
		)( X_s, Y_s, Z_s )
		\|_\V
		\, ds
		< \infty 
) = 1 $, 
		\item \label{eq:W_defined}
		it holds
		that 
		$ \P(
		\int_{t_0}^T
		\|
		\varphi'( S_{ s, T } X_s )
		S_{ s, T } Z_s
		\|_{ \gamma(U, \V ) }^2 \,
		ds < \infty) = 1 $,
		and
		\item\label{eq:mild_ito_type_independent}
		it holds
		that  
		\begin{equation}
		\begin{split}
		[ \varphi( \bar{X}_\tau ) ]_{\P, \mathcal{B}( \V ) }
	 & = 
		[
		\varphi(
		S_{ t_0, T }
		X_{ t_0 }
		)
		]_{\P, \mathcal{B}( \V ) }
		+
		\int_{ t_0 }^\tau
		(
		\mathcal{L}^S_{ s, T } \varphi
		)( X_s, Y_s, Z_s )
		\, { \bf ds }
		\\
		&
		\quad
		+
		\int_{ t_0 }^\tau
		\varphi'(
		S_{ s, T }
		X_s
		) \,
		S_{ s, T } \,
		Z_s \, dW_s
		.
		\end{split}
		\end{equation}
	\end{enumerate}
\end{corollary}
The next result, Corollary~\ref{cor:itoauto2},
specializes Corollary~\ref{cor:itoauto}
to the case where
$ \forall \, \omega \in \Omega \colon
\tau(\omega ) = T $.
Corollary~\ref{cor:itoauto2}
is an immediate consequence of
Corollary~\ref{cor:itoauto},
Lemma~\ref{lemma:transformation},
and Lemma~\ref{lemma:transformation2}.
\begin{corollary}
	\label{cor:itoauto2}
	Assume the setting in Subsection~\ref{sec:setting_mild_calculus},
	let
	$
	X \colon [ t_0, T ] \times \Omega
	\rightarrow V
	$
	be a mild It{\^o} process
	with evolution family
	$
	S \colon \angle
	\rightarrow L( \hat{V}, \check{V} )
	$,
	mild drift
	$
	Y \colon [ t_0, T ] \times
	\Omega \rightarrow \hat{V}
	$, 
	and mild
	diffusion
	$
	Z \colon [ t_0, T ] \times
	\Omega \rightarrow
	\gamma(U,\hat{V})
	$, 
	and let
	$
	\varphi \in C^2(
	\check{V}, \V
	)
	$. 
	Then 
	\begin{enumerate}[(i)]
	\item
	it holds that
	$ \P( 
		\int_{ t_0 }^T
		\|
			(
			\mathcal{L}^S_{ s, T } \varphi
			)( X_s, Y_s, Z_s )
			\|_\V
			\, ds
			< \infty 
	) = 1 $, 
		\item
		it holds
		that 
		$ \P( \int_{t_0}^T \|
		\varphi'( S_{ s, T } X_s )
		S_{ s, T } Z_s
		\|_{ \gamma(U, \V ) }^2 \, ds
		< \infty ) = 1 $,
		\item 
		\label{item:terminal_well}
		it holds that
		$ \P (  X_T \in \check{V} ) = 1 $,
		and
		\item
		it holds
		that 
		\begin{equation}  
		\begin{split}
			\label{eq:mild_ito_type_independent2}
		\big[ \varphi( X_T \mathbbm{1}_{ \{ X_T \in \check{V} \} } ) \big]_{\P, \mathcal{B}( \V ) }
		&
		=
		[
		\varphi(
		S_{ t_0, T }
		X_{ t_0 }
		)
		]_{\P, \mathcal{B}( \V ) }
		+
		\int_{ t_0 }^T
		(
		\mathcal{L}^S_{ s, T } \varphi
		)( X_s, Y_s, Z_s )
		\, { \bf ds }
		\\
		&
		\quad
		+
		\int_{ t_0 }^T
		\varphi'(
		S_{ s, T }
		X_s
		) 
		S_{ s, T } \,
		Z_s \, dW_s
		.
		\end{split}
		\end{equation}
	\end{enumerate}
\end{corollary}
\subsection{Mild Dynkin-type formula}
Under suitable additional assumptions (see Corollary~\ref{cor:mild_dynkin_formula} below),
the stochastic integral 
in \eqref{eq:mild_ito_type_independent} is integrable 
and centered.
This is the subject of the following result.

\begin{corollary}[Mild Dynkin-type formula]
	\label{cor:mild_dynkin_formula}
	Assume the setting in Subsection~\ref{sec:setting_mild_calculus},
	let
	$
	X \colon [ t_0, T ] \times \Omega
	\rightarrow V
	$
	be a mild It{\^o} process
	with evolution family
	$
	S \colon \angle
	\rightarrow L( \hat{V}, \check{V} )
	$,
	mild drift
	$
	Y \colon [ t_0, T ] \times
	\Omega \rightarrow \hat{V}
	$, 
	and mild
	diffusion
	$
	Z \colon [ t_0, T ] \times
	\Omega \rightarrow
	\gamma(U,\hat{V})
	$, 
	let
	$
	\bar{X} \colon [ t_0, T ] \times \Omega \rightarrow \check{V}
	$ 
	be a stochastic process
	with continuous sample paths
	which satisfies
	$
	\forall \, t \in [ t_0, T ) \colon
	\P\big(
	\bar{X}_t
	=
	S_{ t, T } X_t
	\big) = 1
	$
	(see Lemma~\ref{lemma:transformation}),
	let
	$
	\varphi \in C^2(
	\check{V}, \V
	)
	$,
	and let 
	$ \tau \colon \Omega \rightarrow [ t_0, T ] $ 
	be an
	$
	\mathbb{F}
	$-stopping time which satisfies
	that 
	$
	\ES\big[
	|
	\int_{ t_0 }^\tau
	\|
	\varphi'(
	S_{ s, T }
	X_s
	) \,
	S_{ s, T } \,
	Z_s
	\|^2_{ \gamma( U, \V ) }
	\, ds
	|^{1/2}
	\big]
	+
	\min\!\big\{
	\ES\big[
	\|
	[ \varphi( S_{ t_0, T } X_{ t_0 } ) ]_{ \P, \B( \V ) }
	+ 
	\int_{ t_0 }^\tau
	(
	\mathcal{L}^S_{ s, T } \varphi
	)( X_s, Y_s, Z_s )
	\, { \bf ds }
	\|_\V
	\big]
	$,
	$
	\ES\big[
	\| 
	\varphi(
	\bar{X}_{ \tau }
	)
	\|_\V
	\big]
	\big\}
	< \infty
	$. 
	Then
	\begin{enumerate}[(i)]
		\item \label{item:trivial_step}
		it holds that
		$
		\ES\big[
		\| \varphi(\bar{X}_\tau) \|_\V
		\big]
		+
		\E \big[
		\|
		[ \varphi( S_{ t_0, T } X_{ t_0 } ) ]_{\P, \B( \V ) }
		+ 
		\int_{ t_0 }^\tau
		(
		\mathcal{L}^S_{ s, T } \varphi
		)( X_s, Y_s, Z_s )
		\, { \bf ds }
		\|_\V
		\big]
		< \infty
		$
		and 
		\item
		it holds that
		\begin{equation}
		\label{eq:mild_dynkin_formula}
		\begin{split}
		&
		\ES\big[
		\varphi( \bar{X}_\tau )
		\big]
		=
		\ES\Big[
		[ \varphi(
		S_{ t_0, T }
		X_{ t_0 }
		)
		]_{\P, \B(\V)}
		+
		\smallint\limits_{ t_0 }^\tau
		(
		\mathcal{L}^S_{ s, T } \varphi
		)( X_s, Y_s, Z_s )
		\, { \bf ds }
		\Big]
		.
		\end{split}
		\end{equation}
	\end{enumerate}
\end{corollary}
\begin{proof}[Proof of Corollary~\ref{cor:mild_dynkin_formula}]
	%
	%
	%
	%
	First, note that
	item~\eqref{eq:mild_ito_type_independent}
	of Corollary~\ref{cor:itoauto}
	proves that
	\begin{equation}
	\label{eq:Ito_one}
	\begin{split}
	&
	[ \varphi( \bar{X}_\tau ) ]_{\P, \mathcal{B}( \V ) }
	=
	[
	\varphi(
	S_{ t_0, T }
	X_{ t_0 }
	)
	]_{\P, \mathcal{B}( \V ) }
	\\
	& 
	+
	\int_{ t_0 }^\tau
	(
	\mathcal{L}^S_{ s, T } \varphi
	)( X_s, Y_s, Z_s )
	\, { \bf ds }
	+
	\int_{ t_0 }^\tau
	\varphi'(
	S_{ s, T }
	X_s
	) \,
	S_{ s, T } \,
	Z_s \, dW_s
	.
	\end{split}
	\end{equation}
	Moreover, the fact that 
	$ \int_{t_0}^{ \min \{ t, \tau \} } \varphi'(S_{s,T} X_s ) S_{s,T} Z_s \, dW_s $,
	$ t \in [t_0,T] $,
	is a local $ \mathbb{F} $-martingale,
	the assumption that
	$ \ES\big[
	|
	\int_{ t_0 }^\tau
	\|
	\varphi'(
	S_{ s, T }
	X_s
	) \,
	S_{ s, T }  
	Z_s
	\|^2_{ \gamma( U, \V ) }
	\, ds
	|^{1/2}
	\big] <  \infty $,
	and, e.g., 
	the Burkholder-Davis-Gundy type inequality 
	in
	Van Neerven et al.\ \cite[Theorem~4.7]{VanNeerven2015}
	ensure that
	\begin{equation} 
	\int_{t_0}^{ \min \{ t, \tau \} } \varphi'(S_{s,T} X_s ) S_{s,T} Z_s \, dW_s, \quad t \in [ t_0, T ]
	,
	\end{equation}
	is an $ \mathbb{F} $-martingale. 
	This, the fact that
	\begin{equation}
	\min\!\big\{
	\E\big[
	\|
	[ \varphi( S_{ t_0, T } X_{ t_0 } ) ]_{\P, \B(\V)}
	+ 
	\smallint\nolimits_{ t_0 }^\tau
	(
	\mathcal{L}^S_{ s, T } \varphi
	)( X_s, Y_s, Z_s )
	\, { \bf ds}
	\|_\V
	\big],
	\ES\big[
	\| 
	  \varphi(
	\bar{X}_{ \tau }
	)
	\|_\V
	\big]
	\big\}
	< \infty
	,
	\end{equation}
	and~\eqref{eq:Ito_one}
	prove that item~\eqref{item:trivial_step} holds
	and that
	\begin{equation} 
	\begin{split}
	&
	\E[ \varphi( \bar{X}_{ \tau } ) ]
	=
	\E\bigg[
[	\varphi(
	S_{ t_0, T }
	X_{ t_0 }
	) 	]_{\P, \B(\V)}
	+
	\int_{ t_0 }^{ \tau }
	(
	\mathcal{L}^S_{ s, T } \varphi
	)( X_s, Y_s, Z_s )
	\, { \bf ds }
	\bigg]
	.
	\end{split}
	\end{equation}
	The proof of Corollary~\ref{cor:mild_dynkin_formula}
	is thus completed.
\end{proof}
%
%
%
%
\subsection{Weak estimates for terminal values of mild It\^{o} processes}
\label{sec:weak_estimates_terminal}

\begin{proposition}
	\label{prop:mild_ito_stopping_limit}
	Assume the setting in Subsection~\ref{sec:setting_mild_calculus},
	let
	$
	X \colon [ t_0, T ] \times \Omega
	\rightarrow V
	$
	be a mild It{\^o} process
	with evolution family
	$
	S \colon \angle
	\rightarrow L( \hat{V}, \check{V} )
	$,
	mild drift
	$
	Y \colon [ t_0, T ] \times
	\Omega \rightarrow \hat{V}
	$, 
	and mild
	diffusion
	$
	Z \colon [ t_0, T ] \times
	\Omega \rightarrow
	\gamma(U,\hat{V})
	,
	$
	let
	$
	\varphi \in C^2(
	\check{V}, \V
	)
	,
	$
	and assume that 
	$
	\big\{
	\| 
	\varphi(
	[
	S_{ t_0, T } \, X_{ t_0 }
	]_{\P, \B( \check{V} )}
	+
	\int^\tau_{ t_0 }
	S_{ s, T } \, Y_s 
	\, { \bf ds}
	+
	\int^\tau_{ t_0 }
	S_{ s, T } \, Z_s 
	\, dW_s
	) 
	\|_\V
	\colon
	\mathbb{F}\text{-stopping time }
	\tau \colon \Omega \rightarrow [ t_0, T ]
	\big\}
	$ 
	is uniformly $ \P $-integrable.
	Then
	\begin{enumerate}[(i)]
		\item \label{item_well_defined}
		it holds that
		$ \P ( X_T \in \check{V} ) = 1 $,
		\item \label{item:integrability} it holds that 
		$
		\ES\big[
		\big\|
		\varphi (X_T \mathbbm{1}_{ \{ X_T \in \check{V} \} } )
		\big\|_\V
		+
		\|
		\varphi( S_{ t_0, T } X_{ t_0 } )
		\|_\V
		\big]
		< \infty
		$, 
		and 
		\item it holds that
		\begin{equation}
		\label{eq:mild_ito_stopping_limit}
		\begin{split}
		&
		\big\|
		\ES\big[
		\varphi ( X_T \mathbbm{1}_{ \{ X_T \in \check{V} \} } )
		\big]
		\big\|_\V
		\leq
		\big\|
		\ES\big[
		\varphi( S_{ t_0, T } X_{ t_0 } )
		\big]
		\big\|_\V
		+ 
		\smallint_{ t_0 }^T
		\E \big[
		\|
		(
		\mathcal{L}^S_{ s, T } \varphi
		)( X_s, Y_s, Z_s )
		\|_\V
		\big]
		\, ds
		.
		\end{split}
		\end{equation}
	\end{enumerate}
\end{proposition}

\begin{proof}[Proof of Proposition~\ref{prop:mild_ito_stopping_limit}]
	Throughout this proof let
	$
	\tau_n \colon \Omega \rightarrow [ t_0, T ]
	$, 
	$ n \in \N $, 
	be the functions which satisfy for all $ n \in \N $ that 
	\begin{equation}
	\label{eq:taun_def}
	\tau_n
	=
	\inf\!\left(
	\{ T \}
	\cup
	\left\{
	t \in [ t_0, T ] \colon
	\smallint_{ t_0 }^t
	\|
	\varphi'(
	S_{ s, T }
	X_s
	) \,
	S_{ s, T } \,
	Z_s
	\|^2_{ \gamma( U, \V ) }
	\, ds
	\geq n
	\right\}
	\right)
	\end{equation}
	and let 
	$
	\bar{X} \colon [ t_0, T ] \times \Omega \rightarrow \check{V}
	$ 
	be a stochastic process with continuous sample paths
	which satisfies
	\begin{equation} 
	\forall \, 
	t \in [ t_0, T ) 
	\colon
	\P\big(
	\bar{X}_t
	=
	S_{ t, T } X_t
	\big) = 1 
	\end{equation}
	(cf.\ Lemma~\ref{lemma:transformation}).
	Note that item~\eqref{item:terminal_well}
	of Corollary~\ref{cor:itoauto2}
	establishes item~\eqref{item_well_defined}.
	Moreover, observe that 
	the assumption that the set
	$
	\big\{
	\| 
	\varphi(
	[ S_{ t_0, T } \, X_{ t_0 } ]_{\P, \B(\check{V})}
	+
	\int^\tau_{ t_0 }
	S_{ s, T } \, Y_s 
	\, { \bf ds }
	+
	\int^\tau_{ t_0 }
	S_{ s, T } \, Z_s 
	\, dW_s
	) 
	\|_\V
	\colon
	$
	$
	\mathbb{F}\text{-stopping time }
	\tau \colon \Omega \rightarrow [ t_0, T ]
	\big\}
	$ 
	is uniformly $ \P $-integrable
	proves
	item~\eqref{item:integrability}.
	%
	%
	%
	%
	Next note that item~\eqref{eq:W_defined}
	of
	Corollary~\ref{cor:itoauto} shows that  
	\begin{equation} 
	\P\bigg(
	\int_{ t_0 }^T
	\|
	\varphi'(
	S_{ s, T }
	X_s
	) \,
	S_{ s, T } \,
	Z_s
	\|^2_{ \gamma( U, \V ) }
	\, ds
	< \infty \bigg) = 1
	.
	\end{equation}
	This establishes that 
	\begin{equation} 
	\label{eq:NR}
	\P\Big( 
	\lim_{ n \rightarrow \infty } \tau_n
	=
	T \Big) = 1
	.
	\end{equation}
	In addition, note that  
	Lemma~\ref{lemma:transformation2} and
	the assumption that the set 
	$
	\big\{ 
	\|
	\varphi( 
	[ S_{ t_0, T } \, X_{ t_0 } ]_{\P, \B( \check{V} ) }
	+
	\int^\tau_{ t_0 }
	S_{ s, T } \, Y_s 
	\, { \bf ds}
	+
	\int^\tau_{ t_0 }
	S_{ s, T } \, Z_s 
	\, dW_s
	)
	\|_\V
	\colon
	\mathbb{F}\text{-stopping} $ $ \text{time }
	\tau \colon \Omega \rightarrow [ t_0, T ]
	\big\}
	$
	is uniformly $ \P $-integrable
	ensure that the set
	$  
	\{
	\| 
	\varphi( \bar{X}_{ \tau_n } )
	\|_\V
	\colon
	n \in \N
	\}
	$
	is uniformly $ \P $-integrable.
	Equation~\eqref{eq:taun_def}
	hence shows that for all $ n \in \N $ 
	it holds that
	\begin{equation} 
	\ES\big[
	\|
	\varphi( \bar{X}_{ \tau_n } )
	\|_\V
	\big]
	+
	\ES\bigg[
	\int_0^{ \tau_n }
	\|
	\varphi'( S_{ s, T } X_s) S_{ s, T } Z_s )
	\|_{ \gamma( U, \V ) }^2
	\,
	ds
	\bigg]
	< \infty
	.
	\end{equation}
	The fact that
	for all $ n \in \N $ it holds that
	$ \tau_n $ is an $ \mathbb{F} $-stopping time
	thus allows us to apply Corollary~\ref{cor:mild_dynkin_formula}
	to obtain that for all $ n \in \N $
	it holds that
	\begin{equation}
	\begin{split} 
	\ES\big[
	\varphi( \bar{X}_{ \tau_n } )
	\big]
	&
	=
	\ES\bigg[
	[ \varphi(
	S_{ t_0, T }
	X_{ t_0 }
	)
	]_{\P, \B( \V ) }
	+
	\int_{ t_0 }^{ \tau_n }
	(
	\mathcal{L}^S_{ s, T } \varphi
	)( X_s, Y_s, Z_s )
	\, { \bf ds }
	\bigg]
	\\
	&
	=
	\ES[
	\varphi(
	S_{ t_0, T }
	X_{ t_0 }
	)
	] 
	+
	\E \bigg[
	\int_{ t_0 }^{ \tau_n }
	(
	\mathcal{L}^S_{ s, T } \varphi
	)( X_s, Y_s, Z_s )
	\, { \bf ds }
	\bigg]
	.
	\end{split}
	\end{equation}
	The triangle inequality hence proves that 
	\begin{equation}
	\limsup_{ n \to \infty }
	\big\|
	\ES\big[
	\varphi( \bar{X}_{ \tau_n } )
	\big]
	\big\|_\V
	\leq
	\big\|
	\ES\big[
	\varphi(
	S_{ t_0, T }
	X_{ t_0 }
	)
	\big]\|_\V  
	+
	\smallint\nolimits_{ t_0 }^T
	\ES\big[
	\|
	(
	\mathcal{L}^S_{ s, T } \varphi
	)( X_s, Y_s, Z_s )
	\|_\V
	\big]
	\, ds
	.
	\end{equation}
	This together with~\eqref{eq:NR},
	item~\eqref{item:transfortmation2} of
	Lemma~\ref{lemma:transformation2},
	and the 
	uniform $ \P $-integrability of 
	$
	\{
	\| 
	\varphi( \bar{X}_{ \tau_n } )
	\|_\V
	\colon
	n \in \N
	\}
	$
	assures~\eqref{eq:mild_ito_stopping_limit}.
	The proof of Proposition~\ref{prop:mild_ito_stopping_limit}
	is thus completed.
\end{proof}

\begin{proposition}[Test functions with at most polynomial growth]
	\label{prop:mild_ito_maximal_ineq}
	Assume the setting in Subsection~\ref{sec:setting_mild_calculus}, 
	let
	$
	X \colon [ t_0, T ] \times \Omega
	\rightarrow V
	$
	be a mild It{\^o} process
	with evolution family
	$
	S \colon \angle
	\rightarrow L( \hat{V}, \check{V} )
	$,
	mild drift
	$
	Y \colon [ t_0, T ] \times
	\Omega \rightarrow \hat{V}
	$, 
	and mild
	diffusion
	$
	Z \colon [ t_0, T ] \times
	\Omega \rightarrow
	\gamma(U,\hat{V})
	$, 
	and let 
	$ p \in [ 0, \infty ) $, 
	$
	\varphi \in C^2(
	\check{V}, \V
	)
	$ 
	satisfy 
	$
	\sup_{ x \in \check{V} }
	\big[
	\| \varphi(x) \|_\V
	( 1 + \| x \|^p_{ \check{V} } )^{ - 1 }
	\big]
	< \infty
	$
	and
	$
	\E \big[
	|\!
	\int^T_{ t_0 }
	\| S_{ s, T } Z_s \|^2_{ \gamma( U, \check{V} ) }
	\, ds
	|^{ p / 2 } 
	+
	\| S_{ t_0, T } X_{ t_0 } \|_{ \check{V} }^p
	+
	| \! \int^T_{ t_0 }
	\| S_{ s, T } Y_s \|_{ \check{V} }
	\, ds
	|^p 
	\big] < \infty $.
	Then 
	\begin{enumerate}[(i)]
		\item it holds that
		$ \P ( X_T \in \check{V} ) = 1 $, 
		\item
		it holds that
		$
		\ES\big[
		\|
		\varphi ( X_T \mathbbm{1}_{ \{ X_T \in \check{V} \} } )
		\|_\V
		+
		\|
		\varphi( S_{ t_0, T } X_{ t_0 } )
		\|_\V
		\big]
		< \infty
		$, 
		and
		\item it holds that
		\begin{equation}
		\begin{split}
		&
		\big\|
		\ES\big[
		\varphi ( X_T \mathbbm{1}_{ \{ X_T \in \check{V} \} } )
		\big]
		\big\|_\V
		\leq
		\big\|
		\ES\big[
		\varphi( S_{ t_0, T } X_{ t_0 } )
		\big]
		\big\|_\V
		+
		\smallint_{ t_0 }^T
		\E \big[
		\|
		(
		\mathcal{L}^S_{ s, T } \varphi
		)( X_s, Y_s, Z_s )   
		\|_\V
		\big] 
		ds
		.
		\end{split}
		\end{equation}
	\end{enumerate}
\end{proposition}
\begin{proof}[Proof of Proposition~\ref{prop:mild_ito_maximal_ineq}]
	Throughout this proof let 
	$
	\bar{X} \colon [ t_0, T ] \times \Omega \rightarrow \check{V}
	$ 
	be a stochastic process 
	with 
	continuous sample paths which satisfies
	$
	\forall \, t \in [ t_0, T ) \colon
	\P\big(
	\bar{X}_t
	=
	S_{ t, T } X_t
	\big) = 1
	$
	(cf.\ Lemma~\ref{lemma:transformation})
	and let
	$ \mathcal{Z} \colon [t_0,T] \times \Omega \to \check{V} $
	be a stochastic process with continuous sample paths
	which satisfies for all
	$ t \in [t_0, T] $ that
	\begin{equation} 
	[ \mathcal{Z}_t ]_{ \P, \mathcal{B}( \check{V} ) } 
	= 
	\int_{t_0}^t 
	S_{s,T} Z_s \, dW_s
	.
	\end{equation}
	Observe that Lemma~\ref{lemma:transformation2} 
	implies that for all $ t \in [ t_0, T ] $ it holds 
	$ \P $-a.s.\ that 
	\begin{equation}
	\label{eq:path_varbar}
	\begin{split}
	&
	\| \varphi( \bar{X}_t ) \|_\V
	\leq
	\Bigg[
	\sup_{ x \in \check{V} }
	\frac{
		\| \varphi(x) \|_\V
	}{
	( 1 + \|x\|^p_{ \check{V} } )
}
\Bigg]
( 
1 + \| \bar{X}_t \|^p_{ \check{V} } 
)
\\ & \leq
3^p 
\Bigg[
\sup_{ x \in \check{V} }
\frac{
	\| \varphi(x) \|_\V
}{
( 1 + \|x\|^p_{ \check{V} } )
}
\Bigg]
\Bigg(
1
+
\|
S_{ t_0, T } X_{ t_0 }
\|^p_{ \check{V} }
+
\bigg|
\int^T_{ t_0 }
\|
S_{ s, T } Y_s
\|_{ \check{V} }
\, ds
\bigg|^p
+
\|
\mathcal{Z}_t
\|^p_{ \check{V} }
\Bigg)
%
.
\end{split}
\end{equation}
Moreover, e.g., 
the Burkholder-Davis-Gundy type inequality 
in
Van Neerven et al.\ \cite[Theorem~4.7]{VanNeerven2015}
shows that there exists a real number 
$ C \in [ 1, \infty ) $ such that 
%
%
\begin{equation}
\label{eq:maximal_ineq}
\begin{split}
&
\E\bigg[
\sup_{ t \in [ t_0, T ] }
\|
\mathcal{Z}_t
\|^p_{ \check{V} }
\bigg]
\leq
C
\,
\E\Bigg[
\bigg|
\int^T_{ t_0 }
\|
S_{ s, T } \, Z_s
\|^2_{ \gamma( U, \check{V} ) }
\, ds
\bigg|^{ p / 2 }
\Bigg]
.
\end{split}
\end{equation}
Combining \eqref{eq:path_varbar} and \eqref{eq:maximal_ineq} yields that there exists a real number 
$ C \in [ 1, \infty ) $ such that 
\begin{equation}
\label{eq:varphi_bound}
\begin{split}
&
\E\bigg[
\sup_{ t \in [ t_0, T ] }
\|
\varphi( \bar{X}_t )
\|_\V
\bigg]
\leq
C
\Bigg(
1
+
\E[
\|
S_{ t_0, T } \, X_{ t_0 }
\|^p_{ \check{V} }
]
\\
&
+
\E\Bigg[
\bigg|
\int^T_{ t_0 }
\|
S_{ s, T } \, Y_s
\|_{ \check{V} }
\, ds
\bigg|^p 
\Bigg]
+
\E\Bigg[
\bigg|
\int^T_{ t_0 }
\|
S_{ s, T } \, Z_s
\|^2_{ \gamma( U, \check{V} ) }
\, ds
\bigg|^{ p/2 }
\Bigg]
\Bigg)
.
\end{split}
\end{equation}
The assumption that
$
\E \big[
\| S_{ t_0, T } X_{ t_0 } \|_{ \check{V} }^p
+
| \!
\int^T_{ t_0 }
\| S_{ s, T } Y_s \|_{ \check{V} }
\, ds
|^p 
+
| \!
\int^T_{ t_0 }
\| S_{ s, T } Z_s \|^2_{ \gamma( U, \check{V} ) }
\, ds
|^{ p / 2 } 
\big] < \infty $ 
hence
ensures that
\begin{equation} 
\E \bigg[
\sup_{ t \in [ t_0, T ] }
\|
\varphi( \bar{X}_t )
\|_V
\bigg ]
< \infty
.
\end{equation} 
Lemma~\ref{lemma:transformation2}
therefore proves that
\begin{equation}
\ES\bigg[
\sup_{ t \in [ t_0, T ] }
\bigg\|
\varphi
\bigg( 
S_{ t_0, T } X_{ t_0 }
+
\int_{ t_0 }^t
S_{ s, T } Y_s 
\mathbbm{1}_
{
	\{
	\int_{t_0}^T \| S_{ r, T } Y_r \|_{\check{V}} \, dr
	< \infty \}
	}
\, ds
+
\mathcal{Z}_t
\bigg)
\bigg\|_\V
\bigg]
< \infty
.
\end{equation}
Combining this with Proposition~\ref{prop:mild_ito_stopping_limit}
completes the proof of 
Proposition~\ref{prop:mild_ito_maximal_ineq}.
\end{proof}
\subsubsection*{Acknowledgements}
This project has been supported through the SNSF-Research project $ 200021\_156603 $ ``Numerical 
approximations of nonlinear stochastic ordinary and partial differential equations''.

\bibliographystyle{acm}
\bibliography{bibfile}
\end{document}